\documentclass[11pt,twoside]{amsart}
\usepackage[latin1]{inputenc}
\usepackage[T1]{fontenc}
\usepackage{mathtools}
\usepackage{graphicx}
\usepackage{tikz}
\usepackage{enumerate}
\usetikzlibrary{chains}

\usepackage[latin1]{inputenc}
\usepackage{amsmath}
\usepackage{amsthm}
\usepackage{amssymb}

\usepackage[all]{xy}
\usepackage{hyperref}
\newtheorem{thm}{Theorem}[section]

\newtheorem{prop}[thm]{Proposition}

\newtheorem{lem}[thm]{Lemma}
\newtheorem{cor}[thm]{Corollary}

\numberwithin{equation}{section}

\theoremstyle{definition}
\newtheorem{definition}[thm]{Definition}
\newtheorem{remark}[thm]{Remark}
\newtheorem{ex}[thm]{Example}

\newcommand{\Db}{{\rm D}^{\rm b}}

\newcommand{\Pic}{{\rm Pic}}
\newcommand{\rk}{{\rm rk}}

\newcommand{\Hom}{{\rm Hom}}

\newcommand{\cal}{\mathcal}
\newcommand{\ka}{{\cal A}}

\newcommand{\kc}{{\cal C}}

\newcommand{\ki}{{\cal I}}

\newcommand{\km}{{\cal M}}
\newcommand{\ko}{{\cal O}}
\newcommand{\kp}{{\cal P}}

\newcommand{\LL}{\mathbb{L}}
\newcommand{\ZZ}{\mathbb{Z}}
\newcommand{\QQ}{\mathbb{Q}}
\newcommand{\RR}{\mathbb{R}}
\newcommand{\CC}{\mathbb{C}}

\newcommand{\FF}{\mathbb{F}}

\newcommand{\PP}{\mathbb{P}}

\newcommand{\cctwotilde}{\,\,\tilde{\phantom{Q.\!\!}}\!\!{\tilde{\!\kc}}}

\newcommand{\OO}{{\rm O}}
\renewcommand{\to}{\xymatrix@1@=15pt{\ar[r]&}}
\renewcommand{\rightarrow}{\xymatrix@1@=15pt{\ar[r]&}}
\renewcommand{\leftarrow}{\xymatrix@1@=15pt{&\ar[l]}}
\renewcommand{\mapsto}{\xymatrix@1@=15pt{\ar@{|->}[r]&}}
\renewcommand{\twoheadrightarrow}{\xymatrix@1@=18pt{\ar@{->>}[r]&}}
\renewcommand{\hookrightarrow}{\xymatrix@1@=15pt{\ar@{^(->}[r]&}}
\newcommand{\hook}{\xymatrix@1@=15pt{\ar@{^(->}[r]&}}
\newcommand{\congpf}{\xymatrix@1@=15pt{\ar[r]^-\sim&}}
\renewcommand{\cong}{\simeq}

\usepackage{comment}
\excludecomment{comment}

\begin{document}

\title[]{Hodge theory of cubic fourfolds, their Fano varieties, and associated K3 categories}

\author[D. Huybrechts]{Daniel Huybrechts}

\address{Mathematisches Institut,
Universit{\"a}t Bonn, Endenicher Allee 60, 53115 Bonn, Germany}
\email{huybrech@math.uni-bonn.de}

\begin{abstract} \noindent
These are notes of lectures given at the school `Birational Geometry of Hypersurfaces' in Gargnano in March 2018.
The main goal was to discuss the Hodge structures that come naturally associated with a cubic fourfold. 
The emphasis is on the Hodge and lattice theoretic aspects with many technical details worked out explicitly.
More geometric or derived results are only hinted at.

 \vspace{-2mm}
\end{abstract}

\maketitle
{\let\thefootnote\relax\footnotetext{The author is supported by the SFB/TR 45 `Periods,
Moduli Spaces and Arithmetic of Algebraic Varieties' of the DFG
(German Research Foundation) and the Hausdorff Center for Mathematics.}
\marginpar{}
}

The  primitive Hodge structure of a smooth cubic fourfold $X\subset \PP^5$ is  concentrated in degree four and it is of a very particular type.
Once a Tate twist is applied and the sign of the intersection form is changed, it reveals its true nature. It very much looks like the Hodge structure of
a K3 surface. In his thesis Hassett \cite{HassComp} studied this curious relation  and the intricate lattice theory behind it in greater detail. He established a transcendental correspondence between polarized K3 surfaces of certain degrees and special cubic fourfolds, some
aspects of which are reminiscent of the Kuga--Satake construction. The geometric nature of the Hassett correspondence is still not completely understood but it seems that derived categories are central for its understanding.
Work of Addington and Thomas \cite{AT} represents an important step in this direction, combining Hassett's Hodge theory with
Kuznetsov's categorical approach to hypersurfaces.

The aim of the lectures was to discuss the Hodge structures $H^4(X,\ZZ)$, $H^4(X,\ZZ)_{\rm pr}$, $\widetilde H(X,\ZZ)$, and $H^2(F(X),\ZZ)$,
all naturally associated with a cubic fourfold $X$, and their relation to the Hodge structures $H^2(S,\ZZ)$, $H^2(S,\ZZ)_{\rm pr}$,
$\widetilde H(S,\ZZ)$, and $\widetilde H(S,\alpha,\ZZ)$ that come with a (polarized, twisted) K3 surface $S$.
For a discussion of more motivic aspects, partially covered by the original lectures, and of derived aspects, not touched upon
at all, we have to refer to the existing literature. Most of the content of the lectures is also covered
by \cite{HuyCub}.

\smallskip
\noindent
{\bf Acknowledgements:} I wish to thank Andreas Hochenegger and Paolo Stellari for the organization of the school and  PS
for gently insisting that I should write up these notes. The many questions of the participants have been stimulating and helped me to
improve the quality of the notes. Special thanks to Emma Brakkee, who also went through a first draft and pointed
out many inaccuracies, and Pablo Magni.

\section{Lattice and Hodge theory for cubic fourfolds and K3 surfaces}
In the first section, we collect all facts from Hodge and lattice theory relevant for the study of cubic fourfolds. The 
curious relation between the lattice theory of cubic fourfolds and K3 surfaces has been systematically studied first by Hassett \cite{HassComp}.
Earlier results in this direction are due to Beauville and Donagi \cite{BD}.

\subsection{}
As abstract lattices, the middle cohomology and the
primitive cohomology of a smooth cubic fourfold $X\subset\PP^5$ are  described by 
\begin{eqnarray*}
H^4(X,\ZZ)&\cong& {\rm I}_{21,2}\cong E_8^{\oplus 2}\oplus U^{\oplus 2}\oplus {\rm I}_{3,0},\\
H^4(X,\ZZ)_{\rm pr}&\cong &E_8^{\oplus 2}\oplus U^{\oplus 2}\oplus A_2,
\end{eqnarray*}
where the square of the hyperplane class $h$ is given as $h^2=(1,1,1)\in {\rm I}_{3,0}$. Here,
we use the common notation $E_8$ and $U$ for the unique, unimodular, even lattices of signature
$(8,0)$ and $(1,1)$, respectively, and ${\rm I}_{m.n}$ for the unique, unimodular, odd lattice of signature $(m,n)$,
see \cite[Sec.\ 1.1.5]{HuyCub} for details and references.
It will be convenient to change the sign and introduce the \emph{cubic lattice}
and the \emph{primitive cubic lattice} as
\begin{eqnarray*}
\bar\Gamma &\coloneqq &{\rm I}_{2,21}\cong E_8(-1)^{\oplus 2}\oplus U^{\oplus 2}\oplus {\rm I}_{0,3}\cong H^4(X,\ZZ)(-1),\\
\Gamma &\coloneqq& E_8(-1)^{\oplus 2}\oplus U^{\oplus 2}\oplus A_2(-1)
\cong H^4(X,\ZZ)_{\rm pr}(-1).
\end{eqnarray*}
In particular, from now on $(h^2)^2=-3$. The twist should not be confused
with the Tate twist of the Hodge structure.
It turns out that $E_8(-1)^{\oplus 2}$, certainly the most interesting part of these lattices, will hardly play any role in our discussion.
We shall henceforth abbreviate it by $$E\coloneqq E_8(-1)^{\oplus 2}$$
and consequently write $$\bar\Gamma \cong E\oplus U^{\oplus 2}\oplus {\rm I}_{0,3}\text{ and }\Gamma \cong E\oplus U^{\oplus 2}\oplus A_2(-1).$$
Although there is a priori no geometric reason why K3 surfaces should enter the picture at all, their intersection form will play a central role
in our discussion. We will first address  this first purely on the level of abstract lattice  theory and later add Hodge structures. 

Recall that for a complex K3 surface $S$, its middle
cohomology with the intersection form is the lattice 
$$H^2(S,\ZZ)\cong E\oplus U^{\oplus 3}\cong E\oplus U_1\oplus U_2\oplus U_3\eqqcolon\Lambda,$$
see \cite[Ch.\ 14]{HuyK3}. The summands $U_i$, $i=1,2,3$, are copies of the hyperbolic plane $U$. Indexing them will make
the discussion more explicit and will help us to avoid ambiguities later on. 

 The full cohomology $H^*(S,\ZZ)$ is also endowed with a unimodular
intersection form. It is customary to introduce a sign in the pairing on $(H^0\oplus H^4)(S,\ZZ)$, which, however, does not change the abstract isomorphism type, for $U\cong U(-1)$. The resulting lattice is the \emph{Mukai lattice}
\begin{eqnarray*}
\widetilde H(S,\ZZ)&\coloneqq& H^2(S,\ZZ)\oplus (H^0\oplus H^4)(S,\ZZ)\cong E\oplus U^{\oplus 3}\oplus U_4\\
&\cong &E\oplus U_1\oplus U_2\oplus U_3\oplus U_4\eqqcolon \widetilde\Lambda.
\end{eqnarray*}
The standard basis of $U$  consists of isotropic vectors $e,f$ with $(e.f)=1$. We shall denote the standard bases in the first three copies of $U$
as $e_i,f_i\in U_i$, $i=1,2,3$. However, in order to take into account  the sign change in the Mukai
pairing, we shall use the convention that $(e_4.f_4)=-1$ and that $e_4=[S]\in H^0(S,\ZZ)$
and $f_4=[x]\in H^4(S,\ZZ)$ with $x\in S$ a point.

Next, we introduce an explicit embedding $A_2\,\hookrightarrow\widetilde\Lambda$.
Here, $A_2=\ZZ\, \lambda_1\oplus\ZZ\,\lambda_2$ is the lattice of rank two given by
the intersection form $\left(\begin{matrix}2&-1\\-1&2\end{matrix}\right)$ and we define 
\begin{equation}\label{eqn:embA2}
A_2\,\hookrightarrow U_3\oplus U_4\subset\widetilde\Lambda
\end{equation} by $\lambda_1\mapsto e_4-f_4$ and $\lambda_2\mapsto e_3+f_3+f_4$.
The orthogonal complement $\langle\lambda_1,\lambda_2\rangle^\perp=A_2^\perp\subset\widetilde\Lambda$ is the lattice
$$A_2^\perp=E\oplus U_1\oplus U_2\oplus A_2(-1),$$
where $A_2(-1)\subset U_3\oplus U_4$ is spanned by $\mu_1\coloneqq e_3-f_3$ and
$\mu_2\coloneqq -e_3-e_4-f_4$ satisfying $(\mu_i)^2=-2$ and $(\mu_1.\mu_2)=1$.

\begin{remark}\label{rem:l1perp}
We observe that $\lambda_1^\perp=E\oplus U_1\oplus U_2\oplus U_3\oplus \ZZ(-2)$, where the last direct summand
is generated by $e_4+f_4$. Hence, $\lambda_1^\perp\cong \Lambda\oplus\ZZ(-2)$, which is a lattice of discriminant\footnote{The sign of the discriminant will
be of no importance in our discussion, we tacitly work with its absolute value.} 
 ${\rm disc}(\lambda_1^\perp)=2$ and which contains $A_2^\perp\oplus \ZZ\,(\lambda_1+2\lambda_2)$ as a sublattice of index three.
 As $ H^2(S,\ZZ)\cong \Lambda$
and $H^2(S^{[2]},\ZZ)\cong H^2(S,\ZZ)\oplus\ZZ(-2)$ for the Hilbert scheme $S^{[2]}$ of any K3 surface $S$, this can be
read as a lattice isomorphism $\lambda_1^\perp\cong H^2(S^{[2]},\ZZ)$.
\end{remark}

The discussion so far leads to the fundamental observation that there exists an isomorphism 
$$\bar \Gamma\supset\Gamma\cong A_2^\perp\subset \widetilde\Lambda$$
between the primitive cubic lattice $\Gamma$ and the lattice $A_2^\perp$ inside the Mukai lattice
$\widetilde\Lambda$.

For later use, we record that (\ref{eqn:embA2}) induces inclusions of index three:
$$A_2\oplus A_2(-1)\subset U_3\oplus U_4\text{ and }A_2\oplus A_2^\perp\subset\widetilde\Lambda,$$
where, for example, the quotient of the latter is generated by the image of the class
$(1/3)(\mu_1-\mu_2-\lambda_1+\lambda_2)=e_3+f_4$.

Another technical result that will be crucial at some point later, is the following elementary statement which
is surprisingly difficult to prove, cf.\ \cite[Prop.\ 3.2]{AT}.

\begin{lem}\label{lem:UA2AT}
Consider $A_2\subset \widetilde \Lambda$ as before, let  $U\,\hookrightarrow \widetilde \Lambda$ be an isometric embedding  of a copy of the hyperbolic plane, and
denote by $\overline{A_2+U}$  the saturation of $A_2+U\subset\widetilde\Lambda$. Then
there exists an isometric embedding of a copy of the hyperbolic plane $U'\,\hookrightarrow \overline{A_2+U}$ such that
$\rk(A_2+U')=3$.
\end{lem}

\begin{proof} See \cite{AT} for the proof. 
\end{proof}

\begin{remark} To motivate the notion of Noether--Lefschetz (or Heegner) divisors for cubic fourfolds,
let us recall the corresponding concept for K3 surfaces: For a primitive class $\ell\in\Lambda$ with $(\ell)^2=d$, 
we write $$\Lambda_d\coloneqq\ell^\perp\subset\Lambda.$$ As $\ell$ is in the same $\OO(\Lambda)$-orbit
as the class $e_2+(d/2)\, f_2$, cf.\ \cite[Cor.\ 14.1.10]{HuyK3}, it can abstractly be described as
$$\Lambda_d\cong E\oplus U^{\oplus 2}\oplus \ZZ(-d).$$
It is important to note that the lattices $\Lambda_d$ are in general not contained in $A_2^\perp\subset \widetilde\Lambda$.
\end{remark}

We shall call any primitive vector $v\in\Gamma\cong A_2^\perp$ with $(v)^2<0$ a
\emph{Noether--Lefschetz} vector. With such  Noether--Lefschetz vector one
naturally associates two lattices. On the cubic side,
one defines $$\ZZ \, h^2\oplus \ZZ\, v\subset K_v\subset \bar\Gamma$$
as the saturation of $\ZZ \, h^2\oplus \ZZ \, v\subset\bar\Gamma$. On the K3 side,
we introduce the saturation
$$A_2\oplus \ZZ\,  v\subset L_v\subset \widetilde\Lambda.$$
Note that $L_v$ is of rank three and signature $(2,1)$, while $K_v$ is of rank two and signature $(0,2)$. Clearly, their respective
orthogonal complements are isomorphic: $$ \bar\Gamma\supset K_v^\perp \cong L_v^\perp\subset\widetilde\Lambda,$$ as they are both described as
$v^\perp \subset \Gamma\cong A_2^\perp$. In particular, for the discriminants we have
$$d\coloneqq {\rm disc}(L_v)={\rm disc}(K_v).$$
The situation has been studied in depth in \cite[Prop.\ 3.2.2]{HassComp}:

\begin{lem}[Hassett]\label{lem:HassettKv}
Only the following two cases can occur:
\begin{enumerate}
\item[{\rm (i)}] Either $\ZZ\, h^2\oplus \ZZ \, v=K_v$, $A_2\oplus \ZZ\,  v=L_v$,
and $$d={\rm disc}(K_v)={\rm disc}(L_v)=-3\,(v)^2\equiv 0\, (6)$$
\item[{\rm (ii)}] or $\ZZ\, h^2\oplus \ZZ \, v\subset K_v$, $A_2\oplus \ZZ\,  v\subset L_v$ are both of
index three, and
$$d={\rm disc}(K_v)={\rm disc}(L_v)=-\frac{1}{3}(v)^2\equiv 2\, (6).$$
\end{enumerate}
\end{lem}

\begin{proof}
The main ingredient is the standard formula, see e.g.\ \cite[Sec.\ 14.0.2]{HuyK3},
$${\rm disc}(K_v)\cdot[K_v: \ZZ\, h^2\oplus\ZZ\, v]^2={\rm disc}(\ZZ\, h^2\oplus\ZZ\, v)=-3\, (v)^2.$$
Any  $y\in K_v$ is of the form $y= s\, h^2+t\, v$, with $s,t\in \QQ$.
From $(h.y)\in\ZZ$ one concludes $s\in (1/3)\, \ZZ$ and hence also $t\in(1/3)\, \ZZ$. This shows
that $[K_v: \ZZ\, h^2\oplus\ZZ\, v]=1, 3$, or $=9$, but the last possibility is excluded as $(1/3)\,h^2\not\in\bar\Gamma$.

In the first case, i.e.\ $K_v=\ZZ\, h^2\oplus\ZZ\, v$, one finds
$d={\rm disc}(K_v)=-3\, (v)^2\equiv 0\, (6)$. In the second case, so when the index is three, then
$3\, d=-(v)^2\equiv 0,2,4\,(6)$. On the other hand, $K_v$ admits a basis consisting of 
$h^2$ and another class $x$. Indeed, pick any class $x\in K_v$ whose image generates the quotient $K_v/(\ZZ\, h^2\oplus\ZZ\, v)\cong\ZZ/3\ZZ$.
We may assume $3\, x=s\, h^2+t\, v$ with  $s,t=\pm 1$ and, therefore, $K_v=\ZZ\, h^2\oplus\ZZ\, x$. 
Hence, its discriminant satisfies $d=-3\, (x)^2-(x.h^2)^2\equiv0,2,3,5\,(6)$. Altogether this shows that $d\equiv 0,2\,(6)$.

We claim that $d\equiv0\,(6)$ holds if and only if $K_v=\ZZ\, h^2\oplus\ZZ\, v$.
The `if'-direction' was proven already. For the `only if'-direction, assume that $d\equiv0\,(6)$ but $[K_v: \ZZ\, h^2\oplus\ZZ\, v]=3$.
Pick $x\in K_v$ as above. Then, write  $v=s\, h^2+t\, x$, $s,t\in\ZZ$, and use $(v.h^2)=0$ and the primitivity of $v$
to show $v=r\, ((x.h^2)\, h^2+3\, x)$ with $r=\pm 1,\pm(1/3)$ as $v$ is primitive.
However, $(x.h^2)\equiv0\,(3)$ under the assumption that $d\equiv0\,(6)$. Hence, $\pm v=m\,h^2+x$, $m\in \ZZ$, and, therefore, $x\in \ZZ\, h^2\oplus \ZZ\, v$. This
yields a contradiction and thus proves the assertion.

The assertions for the lattice $L_v$ follows directly from the ones for $K_v$.
\end{proof}

\begin{remark}
Depending on the perspective, it may be useful to study the various cases from the point of view of  $d$ or, alternatively, of $(v)^2$. 
To have the results handy for later use, we restate the above discussion as
$$\begin{array}{lll}d\equiv 0\,(6)~Ê&\Rightarrow&~(v)^2=-d/3\equiv 0\, (6)\text{ or }\equiv \pm 2\, (6),\\
d\equiv 2\, (6)~&\Rightarrow&~Ê(v)^2=-3\,d\equiv0\,(6)\phantom{NMNMMMMHHGs}
\end{array}$$
and\,
$$\begin{array}{lll}
(v)^2\equiv \pm 2\,(6)~&\Rightarrow&~d=-3\,(v)^2\equiv0\,(6),\\
(v)^2\equiv 0\, (6)~&\Rightarrow&~d=-3\,(v)^2\equiv0\, (6)\text{ or } d=-(1/3)\, (v)^2\equiv 2\,(6).
\end{array}$$
In particular, $d$ determines $(v)^2$ uniquely, but not vice versa unless $(v)^2\equiv \pm 2\,(6)$.
\end{remark}

\begin{prop}[Hassett]\label{prop:HassettEichler}
Let $v,v'\in\Gamma$ be two primitive vectors  and assume that ${\rm disc}(L_v)={\rm disc}(L_{v'})$
or, equivalently, ${\rm disc}(K_v)={\rm disc}(K_{v'})$. Then there exist an
orthogonal transformations $g\in \tilde\OO(\Gamma)$ 
such that $g(v)=\pm v'$ and, in particular,
$$L_{v'}\cong L_{g(v)}\text{ and }K_{v'}\cong K_{g(v)}.$$
\end{prop}

The definition of $\tilde\OO(\Gamma)$ will be recalled below.

\begin{proof} We apply Eichler's criterion, cf.\ \cite[Prop.\ 3.3]{Hulek}. If an even lattice $N$ is of the form $N\cong N'\oplus U^{\oplus 2}$, then
a primitive vector $v\in N$ with prescribed $(v)^2\in\ZZ$ and $(1/n)\,\bar v\in A_N$, with $n$ determined by $(v.N)=n\,\ZZ$, is unique up to
the action of $\tilde\OO(N)$. Apply this to $v\in \Gamma\cong A_2^\perp \cong E\oplus U^{\oplus 2}\oplus A_2(-1)$ and use that for any primitive $v\in \Gamma$, either $(v.\Gamma)=\ZZ$ or $=3\,\ZZ$. This follows from $[\bar\Gamma:\Gamma\oplus \ZZ\, h^2]=3$ and the unimodularity of $\bar\Gamma$.

(i) If $(v)^2\equiv 0\,(6)$, there are two cases: Assume first that $d\equiv2\,(6)$ or, equivalently, that $\ZZ\,v\oplus\ZZ\, h^2$  is not saturated.
Then, one finds an element of the form $\alpha\coloneqq (1/3)\, v+ t\, h^2\in \bar\Gamma$. As $(\alpha.w)\in\ZZ$ for
all $w\in\Gamma$, this shows $(v.\Gamma)\subset 3\,\ZZ$. Hence, $n=3$ and $(1/3)\, \bar v=\pm 1\in A_\Gamma\cong\ZZ/3\,\ZZ$.

Assume now that $d\equiv0\,(6)$ and write $v=n_1v_1+ n_2 v_2$ with $v_1\in E\oplus U_1\oplus U_2$ and $v_2\in A_2(-1)$, both primitive, and
$n_1,n_2\in\ZZ$. If $n_1\not\equiv0\,(3)$, then there exists a class $w$ in the unimodular lattice
$E\oplus U_1\oplus U_2\subset\Gamma$ with $(v.w)\not\in 3\, \ZZ$ and hence $(v.\Gamma)=\ZZ$. If $n_1\equiv 0\, (3)$, then
$n_2\not\equiv0\, (3)$, as $v$ is primitive. However, in this case $(1/3)\,(v\pm h^2)=(n_1/3) v_1+(1/3)\,(n_2 v_2\pm  h^2)\in \bar\Gamma$ and so $\ZZ\, v\oplus\ZZ\,h^2$ is not saturated, contradicting $d\equiv0\,(6)$.

\smallskip

(ii) If $(v)^2\equiv \pm2\,(6)$ and hence $(v)^2\not\equiv 0\,(3)$, then $(v.\Gamma)=\ZZ$, $n=1$, and $\bar v\in A_\Gamma$ is trivial.

Hence, in case (i) and (ii), if indeed $d$ and not only $(v)^2$ is fixed, then $(v)^2=(v')^2$ and
$(1/n)\,\bar v=(1/n)\,\bar v'\in A_\Gamma$  (up to sign).
\end{proof}

\begin{remark}\label{rem:epxlitchoices}
Due to the uniqueness, no information is lost when explicit classes $v\in \Gamma\cong A_2^\perp$ are chosen for any given $d$.
In the sequel, we will work with the following ones.

(i) For $d\equiv0\,(6)$, one may choose $v_d\coloneqq e_1-(d/6)\,f_1\in U_1\subset\Gamma$. Observe that indeed,
as explained in the general context above,  $(v_d)^2=-d/3$ and that the lattice $A_2\oplus\ZZ\, v_d$ is saturated
(use $A_2\subset U_2\oplus U_3$ and $v_d\in U_1$), i.e.
$$L_d\coloneqq L_{v_d}=A_2\oplus\ZZ\, v_d.$$ Similarly,
$$K_d\coloneqq K_{v_d}=\ZZ\, h^2\oplus\ZZ\, v_d,$$
which again shows $(v_d)^2=-d/3$. Their orthogonal
complement is
$$\Gamma_d\coloneqq L_d^\perp\cong K_d^\perp \cong E\oplus U_2\oplus A_2(-1)\oplus \ZZ\,( e_1+(d/6)\, f_1)$$
and their discriminant group $$A_{K_d^\perp}\cong A_{K_d}\cong \ZZ/3\ZZ\oplus \ZZ/(d/3)\ZZ$$ is cyclic if and only if $9\nmid d$.

(ii)  For $d\equiv 2\,(6)$, one sets $v_d\coloneqq 3\,(e_1-((d-2)/6)\, f_1)+\mu_1-\mu_2\in U_1\oplus A_2(-1)$. Then both inclusions
$$A_2\oplus\ZZ\, v_d\subset L_d\coloneqq L_{v_d}\text{ and } \ZZ\, h^2\oplus\ZZ\, v_d\subset K_d\coloneqq K_{v_d}$$
are of index three, for example $v_d-\lambda_1+\lambda_2$  and $v_d-h^2$ are divisible by $3$. Use $\lambda_1=e_4-f_4$, $\lambda_2=e_3+f_3+f_4$, $\mu_1=e_3-f_3$, and $\mu_2=-e_3-e_4-f_4$, the latter corresponding to $(1,-1,0),(0,1,-1)\in\ZZ^{\oplus 3}$.
In this case, see \cite{AddCub,HassComp,TanVar}, $$\Gamma_d\coloneqq L_d^\perp\cong K_d^\perp\cong E\oplus U_2\oplus
(\ZZ^{\oplus 3}, (~.~)_A)\text{ with }A\coloneqq\left(\begin{matrix}-2&1&0\\
1&-2&1\\
0&1&(d-2)/3\end{matrix}\right)$$
and $L_d$ and $K_d$ are given by the matrices $-A$ and
$$\left(\begin{matrix}-3&1\\
1&-(d+1)/3\end{matrix}\right),$$
respectively.
The discriminant groups for $d\equiv2\, (6)$ are cyclic, indeed $A_{K_d^\perp}\cong A_{K_d}\cong\ZZ/d\ZZ$.
\end{remark}

In addition to the orthogonal group
\begin{equation}\label{eqn:tildeOO}
\tilde \OO(\Gamma)\coloneqq\{~ g\in\OO(\bar\Gamma)\mid g(h^2)=h^2~\},
\end{equation}
which we will also think of as $\tilde \OO(\Gamma)=\{~g\in \OO(\Gamma)\mid \bar g\equiv {\rm id} \text{ on } A_\Gamma~\}$,
we need to consider
$$\begin{array}{l}
\tilde \OO(\Gamma,K_d)\coloneqq\{~g\in\tilde\OO(\Gamma)\mid g(K_d)=K_d, \text{i.e. } 
g(v_d)=\pm v_d~\}\\
\phantom{HH}\bigcup\\
~~~ \tilde\OO(\Gamma,v_d)\coloneqq\{~g\in\tilde\OO(\Gamma)\mid  g|_{K_d}={\rm id}, \text{i.e. }
 g(v_d)=v_d ~ \}.
\end{array}$$
Observe that  $ \tilde\OO(\Gamma,v_d)$ can be identified with the subgroup of all $g\in\OO(\Gamma_d)$ with
trivial action on the discriminant group $A_{\Gamma_d}\cong A_{K_d}$. Also, by definition, $ \tilde{\rm O}(\Gamma,v_d)\subset \tilde\OO(\Gamma,K_d)$ is a subgroup of index one or two.
Note that the natural homomorphism  $\tilde\OO(\Gamma,K_d)\to \OO(K_d)$ is  neither surjective (let alone injective) nor is its image contained in 
the subgroup of transformations acting trivially on the discriminant $\tilde\OO(K_d)$.

\begin{lem}[Hassett]\label{lem:KDOO} The subgroup $\tilde\OO(\Gamma,v_d)\subset \tilde \OO(\Gamma,K_d)$ is of index at most two. More precisely, one distinguishes the following cases:
\begin{enumerate}
\item[{\rm (i)}]  If $d\equiv0\,(6)$, then $$\tilde\OO(\Gamma,v_d)\subset \tilde \OO(\Gamma,K_d)$$  has index two.
\item[{\rm (ii)}] If $d\equiv2\,(6)$, then $$\tilde\OO(\Gamma,v_d)= \tilde \OO(\Gamma,K_d).$$
\end{enumerate}
\end{lem}

\begin{proof}
(i) According to Lemma \ref{lem:HassettKv}, $d\equiv 0\, (6)$ if and only if $\ZZ\, h^2\oplus \ZZ\, v_d=K_d$, which is contained in ${\rm I}_{0,3}\oplus U_1$.
Let $g\in \tilde\OO(\Gamma)$ be the orthogonal transformation defined by $g={\rm id}$ on $E\oplus U_2\oplus {\rm I}_{0,3}$ and by $g=-{\rm id}$
on $U_1$. Then $g$ is an element in $ \tilde \OO(\Gamma,K_d)\setminus\tilde\OO(\Gamma,v_d)$.

(ii) Now, $d\equiv 2\, (6)$ if and only if $\ZZ\, h^2\oplus \ZZ\, v_d\subset K_d$  has index three
and then $v_d=3\,(e_1-((d-2)/6)\, f_1)+\mu_1-\mu_2$ with $\mu_1=(1,-1,0),\mu_2=(0,1,-1)\in A_2(-1)\subset {\rm I}_{0,3}$ and $h^2=(1,1,1)$.
Now observe that $(1/3)\,(v_d-h^2)\in K_d$,  but $(1/3)\,(-v_d-h^2)\not\in K_d$. 
\end{proof}


\subsection{}  It turns out that certain geometric properties of cubic fourfolds
are encoded by lattice-theoretic properties of  Noether--Lefschetz vectors $v\in\Gamma$.
The following ones are relevant for our purposes. It is a matter of choice, whether they are read
as  conditions on $d$ or on the primitive $v\in\Gamma$. For $d\in \ZZ$ one considers the conditions:\\
$$\begin{array}{ccl}
(\ast)&  \Leftrightarrow&\text{There exists an }L_d.\\[1ex]
(\ast\ast')& \Leftrightarrow&\text{There exists an } L_d \text{ and an embedding }U(n)\,\hookrightarrow L_d\text{ for some }n\ne0.\\[1ex]
(\ast\ast)& \Leftrightarrow&\text{There exists an } L_d\text{ and a primitive embedding }U\,\hookrightarrow L_d.\\[1ex]
(\ast\!\ast\!\ast)& \Leftrightarrow&\text{There exists an } L_d \text{ and a primitive embedding }U\,\hookrightarrow L_d\text{ with }\lambda_1\in U.
\end{array}$$
\medskip

\begin{remark}
\vskip-0.3cm
\smallskip
(i) The following implications trivially hold
$$(\ast\!\ast\!\ast)\Rightarrow(\ast\ast)\Rightarrow(\ast\ast')\Rightarrow(\ast).$$

(ii) Each of the conditions in fact splits in two, distinguishing between $d\equiv 0\, (6)$ and $d\equiv 2\, (6)$. We shall
write accordingly $(\ast)_0$, $(\ast)_2$, $(\ast\ast')_0$,   $(\ast\ast')_2$, etc.
\end{remark}


\begin{lem}\label{lem:OOLamL}
Condition $(\ast\ast)$ holds if and only if  there exists an isomorphism of lattices
$$\varepsilon\colonÊ\Gamma_d\congpf\Lambda_d.$$
In this case, one also has an  isomorphism of groups
$$\tilde\OO(\Gamma,v_d)\cong\tilde\OO(\Lambda_d).$$

\end{lem}
\begin{proof}
Assume that there exists a (primitive) hyperbolic plane $U\,\hookrightarrow L_d$. As the composition with the inclusion
$L_d\subset \widetilde\Lambda$ can  be identified with $U_4\,\hookrightarrow \widetilde\Lambda$ up to the action of $\OO(\widetilde\Lambda)$, 
see \cite[Thm.\ 14.1.12]{HuyK3}, one has $U^\perp\cong\Lambda$. Hence, $\Gamma_d=L_d^\perp \subset U^\perp\cong\Lambda$ is a primitive sublattice of corank one, signature $(2,19)$,
discriminant $d$, and is, therefore, isomorphic to $\Lambda_d$. Conversely, if $L_d^\perp= \Gamma_d\cong\Lambda_d\subset \Lambda\subset\widetilde\Lambda$,
then $U_4\subset L_d$. Here, one again uses that up to $\OO(\widetilde\Lambda)$, there exists only one primitive embedding $\Lambda_d\,\hookrightarrow\widetilde\Lambda$. 

For the isomorphism between the two orthogonal groups, just recall that they are both described as the subgroup of all orthogonal
transformations of $\Gamma_d\cong\Lambda_d$ acting trivially on the discriminant $A_{\Gamma_d}\cong A_{\Lambda_d}\cong\ZZ/d\ZZ$.
\end{proof}

\begin{remark}
As any isometric embedding $U\,\hookrightarrow L_d$ splits, see \cite[Ex.\ 14.0.3]{HuyK3}, one concludes that for $d$ satisfying $(\ast\ast)_0$
and $(\ast\ast)_2$, respectively, that
\begin{eqnarray*}
(\ast\ast)_0\colon&A_2\oplus \ZZ\,v_d\cong  L_d\cong U\oplus \ZZ(d)\text{ and } (v_d)^2=-(1/3)\, d\\
(\ast\ast)_2\colon&ÊA_2\oplus \ZZ\,v_d\,\hookrightarrow  L_d\cong U\oplus \ZZ(d)\text{ index three and } (v_d)^2=-3\,d.
\end{eqnarray*}

\end{remark}

\begin{remark}
For a numerical description of these conditions one needs the following classical facts determining  which numbers are represented by  $A_2$, see \cite{Cox,Kneser}.
(i) For a given even, positive integer $d$ there exists a vector $w\in A_2$ with $(w)^2=d$ if and only if the prime factorization of $d/2$ satisfies
\begin{equation}\label{eqn:A2nonprim}
\xymatrix{\frac{d}{2}=\prod p^{n_p} Ê\text{ with }n_p\equiv0\,(2)\text{ for all } p\equiv2\,(3).}
\end{equation}
(ii) For a given even, positive integer $d$ there exists a primitive vector $w\in A_2$ with $(w)^2=d$ if and only if 
\begin{equation}\label{eqn:A2prim}
\xymatrix{\frac{d}{2}=\prod p^{n_p} Ê\text{ with }n_p=0\text{ for all } p\equiv2\,(3)\text{ and }n_3\leq1.}
\end{equation}
\end{remark}

\begin{prop} Numerically, $(\ast)$, $(\ast\ast')$, $(\ast\ast)$, and  $(\ast\!\ast\!\ast)$  are described by:
%
$$\begin{array}{rccl}
  {\rm (i)}& (\ast) & \Leftrightarrow &d\equiv0,2\, (6).\\
{\rm (ii)}& (\ast\ast') &  \Leftrightarrow &\exists\ w\in A_2\colon (w)^2=d ~~ \Leftrightarrow~~(\ref{eqn:A2nonprim}).\\
{\rm (iii)}&(\ast\ast)& \Leftrightarrow&\exists\ w\in A_2 ~{\rm primitive}\colon (w)^2=d~~ \Leftrightarrow ~(\ref{eqn:A2prim}) ~~\Leftrightarrow~~\exists\, a,n\in\ZZ\colon
d=\frac{2n^2+2n+2}{a}.\\
{\rm (iv)}&(\ast\!\ast\!\ast)&\Leftrightarrow&\exists\, a,n\in\ZZ\colon d=\frac{2n^2+2n+2}{a^2}.
\end{array}$$

%
\end{prop}

\begin{proof}
The first assertion follows from Lemma \ref{lem:HassettKv}. 
\smallskip

To prove (ii), one has to distinguish between the two cases $d\equiv 0\,(6)$ and $d\equiv 2\,(6)$.
Assume first that $(\ast\ast')_0$ holds. Then $L_d=A_2\oplus\ZZ\,v_d$, which contains the isotropic vector $e\in U(n)\subset  L_d$. Writing  $e=w_0+a\, v_d$ for some $w_0\in A_2$ and $a\in \ZZ$, one has $(w_0)^2=a^2d/3$. Hence, $a^2d/6$ satisfies (\ref{eqn:A2nonprim}) and, therefore, $d/2$ does. The latter then yields the existence
of some $w\in A_2$ with $(w)^2=d$. Assume now we are in case $(\ast\ast')_2$, then the standard basis vector $e\in U\subset L_d$ itself might not be contained in $A_2\oplus\ZZ\, v_d$, but $3\,e$ is and replacing
$e$ by $3\,e$ and $(1/3)$ by $3$, one can argue as before.

Conversely, if $d/2$ satisfies (\ref{eqn:A2nonprim}), then we can pick $w\in A_2$ with
$(w)^2=d/3$  for $d\equiv 0\,(6)$ and with $(w)^2=3\,d$ for $d\equiv 2\,(6)$. Then $e\coloneqq w+v_d$ is isotropic. 
Furthermore,  there exists $w'\in A_2$ with $m\coloneqq(e.w')=(w.w')\ne0$. Then $f\coloneqq m \, w'-((w')^2/2)\,e$ satisfies $(f)^2=0$ and $(e.f)=(e.w')^2\eqqcolon n$, which yields an embedding $U(n)\,\hookrightarrow A_2\oplus\ZZ\, v_d\subset L_d$ proving $(\ast\ast')$.

\smallskip

Turning to (iii) and using the notation in (ii), observe that in case $(\ast\ast)_0$, which implies $(\ast\ast')_0$, the class $w_0$ has to be primitive. Indeed, if $w_0=p\,w_1$
for some prime $p$, then $p\mid a$ or $p\mid d/3$. On the other hand, writing $f=w_0'+a'\, v_d$ yields
the contradiction $1=(e.f)=p\,(w_1.w_0')+aa'd/3\equiv0\,(p)$. Hence, $a^2d/6$ satisfies (\ref{eqn:A2prim}) and, therefore, $d/2$ does, i.e.\ there exists
a primitive $w\in A_2$ with $(w)^2=d/2$. The argument for $(\ast\ast)_2$ is similar: If $e=w_0+a\, v_d$,  one argues as before. If  not,
then $3\, e=w_0+a\, v_d$ and if $w_0=p\, w_1$, then $p\ne 3$. All other primes are excluded as before.

For the converse in this situation, we use the arguments above and pick a primitive $w\in A_2$ with $(w)^2=d/3$ or $=3d$, respectively. 
As $A_{A_2}\cong \ZZ/3\ZZ$, either $(w.A_2)=\ZZ$ or $=3\,\ZZ$. If $(w)^2=d/3$, then the former holds (because $3^2\nmid d$) and, therefore,
$w'$ above can be chosen such that $m=1$. Hence, there exists  $U\,\hookrightarrow L_d$. If $(w.A_2)=3\,\ZZ$, so in particular $(w)^2=3\, d$
and  $d\equiv2\,(6)$, then the class  $e\coloneqq w\pm v_d$ is of the form
$e=3\, e'$ with $e'\in L_d$. Therefore, the two classes $e'$ and $f'\coloneqq w'-((w')^2/2)\,e'$, where  $w'\in A_2$ is chosen such that
$(w.w')=3$, define an embedding $U\,\hookrightarrow L_d$.

\smallskip As we will not use the presentation of $d$ as $(2n^2+2n+2)/a$ and $(2n^2+2n+2)/a^2$, respectively, we leave the proof 
of the other equivalences to the reader,
see \cite[Prop.\ 6.1.3]{HassComp} and \cite[Sec.\ 3]{AddCub}.
\end{proof}
The following table lists the first special discriminants, highlighting the difference between the four conditions.
\begin{center}
 \begin{tabular}[t]{|c|c|c|c|c|c|c|c|c|c|c|c|c|c|} \hline
 \raisebox{-0.03cm}{$(\ast\!\ast\!\ast)$}&\raisebox{-0.03cm}{}&\raisebox{-0.03cm}{}&\raisebox{-0.03cm}{14}&\raisebox{-0.03cm}{}&\raisebox{-0.03cm}{} &\raisebox{-0.03cm}{}&\raisebox{-0.03cm}{26}& \raisebox{-0.03cm}{}&\raisebox{-0.03cm}{}&\raisebox{-0.03cm}{}&\raisebox{-0.03cm}{38}&\raisebox{-0.03cm}{42} 
 \\
  [1pt]\hline \raisebox{-0.03cm}{$(\ast\ast)$} &\raisebox{-0.03cm}{}&\raisebox{-0.03cm}{}&\raisebox{-0.03cm}{14}&\raisebox{-0.03cm}{}&\raisebox{-0.03cm}{}&\raisebox{-0.03cm}{}&\raisebox{-0.03cm}{26}&\raisebox{-0.03cm}{}&\raisebox{-0.03cm}{}&\raisebox{-0.03cm}{}&\raisebox{-0.03cm}{38}&\raisebox{-0.03cm}{42} 
  \\
  [1pt]\hline
  \raisebox{-0.03cm}{$(\ast\ast')$}&\raisebox{-0.03cm}{\!8}&\raisebox{-0.03cm}{}&\raisebox{-0.03cm}{14}&\raisebox{-0.03cm}{18}&\raisebox{-0.03cm}{} &\raisebox{-0.03cm}{24}&\raisebox{-0.03cm}{26}& \raisebox{-0.03cm}{}&\raisebox{-0.03cm}{32}&\raisebox{-0.03cm}{}&\raisebox{-0.03cm}{38}&\raisebox{-0.03cm}{42} 
  \\
    [1pt]\hline
 \raisebox{-0.03cm}{$(\ast)$}&\raisebox{-0.03cm}{ \!8 }&\raisebox{-0.03cm}{12}&\raisebox{-0.03cm}{14}&\raisebox{-0.03cm}{18}&\raisebox{-0.03cm}{20} &\raisebox{-0.03cm}{24}&\raisebox{-0.03cm}{26}& \raisebox{-0.03cm}{30}&\raisebox{-0.03cm}{32}&\raisebox{-0.03cm}{36}&\raisebox{-0.03cm}{38}&\raisebox{-0.03cm}{42} 
 \\
[1pt]\hline
 \end{tabular}
\end{center}

\begin{center}
 \begin{tabular}[t]{|c|c|c|c|c|c|c|c|c|c|c|c|c|c|} \hline
 \raisebox{-0.03cm}{$(\ast\!\ast\!\ast)$}&\raisebox{-0.03cm}{}&\raisebox{-0.03cm}{}&\raisebox{-0.03cm}{}&\raisebox{-0.03cm}{}&\raisebox{-0.03cm}{} &\raisebox{-0.03cm}{}&\raisebox{-0.03cm}{62}& \raisebox{-0.03cm}{}&\raisebox{-0.03cm}{}&\raisebox{-0.03cm}{}&\raisebox{-0.03cm}{}&\raisebox{-0.03cm}{}
 \\
  [1pt]\hline \raisebox{-0.03cm}{$(\ast\ast)$} &\raisebox{-0.03cm}{}&\raisebox{-0.03cm}{}&\raisebox{-0.03cm}{}&\raisebox{-0.03cm}{}&\raisebox{-0.03cm}{}&\raisebox{-0.03cm}{}&\raisebox{-0.03cm}{62}&\raisebox{-0.03cm}{}&\raisebox{-0.03cm}{}&\raisebox{-0.03cm}{}&\raisebox{-0.03cm}{74}&\raisebox{-0.03cm}{78} 
  \\
  [1pt]\hline
  \raisebox{-0.03cm}{$(\ast\ast')$}&\raisebox{-0.03cm}{}&\raisebox{-0.03cm}{}&\raisebox{-0.03cm}{50}&\raisebox{-0.03cm}{}&\raisebox{-0.03cm}{} &\raisebox{-0.03cm}{}&\raisebox{-0.03cm}{62}& \raisebox{-0.03cm}{}&\raisebox{-0.03cm}{68}&\raisebox{-0.03cm}{}&\raisebox{-0.03cm}{74}&\raisebox{-0.03cm}{78} 
  \\
    [1pt]\hline
 \raisebox{-0.03cm}{($\ast)$}&\raisebox{-0.03cm}{44}
 &\raisebox{-0.03cm}{48}&\raisebox{-0.03cm}{50}&\raisebox{-0.03cm}{54}&\raisebox{-0.03cm}{56} &\raisebox{-0.03cm}{60}&\raisebox{-0.03cm}{62}& \raisebox{-0.03cm}{66}&\raisebox{-0.03cm}{68}&\raisebox{-0.03cm}{72}&\raisebox{-0.03cm}{74}&\raisebox{-0.03cm}{78}
 \\
[1pt]\hline
 \end{tabular}
\end{center}

\subsection{} 
In the theory of K3 surfaces, there are good reasons to pass from the K3 lattice $\Lambda\cong H^2(S,\ZZ)$
to the Mukai lattice $\widetilde\Lambda \cong \widetilde H(S,\ZZ)\cong H^2(S,\ZZ)\oplus U_4$,
see \cite[Ch.\ 16]{HuyK3} for a survey and references. A similar extension
of lattices, though slightly more technical due to the non-triviality of the canonical bundle,
turns out to be useful for cubics and their comparison with K3 surfaces.

We have already constructed and fixed an isomorphism $\Gamma\cong E\oplus U_1\oplus U_2\oplus A_2(-1)\cong A_2^\perp$,
where  $A_2(-1)\oplus A_2\,\hookrightarrow U_3\oplus U_4$. On the cubic side, one also finds a natural sublattice
isomorphic to $U_3\oplus U_4$, namely $H^{\ast\ne4}(X,\ZZ)$. However, the distinguished $A_2(-1)\subset\Gamma$ sits in $H^4(X,\ZZ)$,
so this has to be modified. Moreover, we will embed $A_2$ into rational cohomology $H^*(X,\QQ)$
and the intersection product on $H^*(X,\QQ)$ is modified by more than a mere sign.

\begin{definition}
The \emph{Mukai pairing} on $H^*(X,\QQ)$ is defined as
\begin{equation}\label{eqn:Mukaipcu}
(\alpha.\alpha')\coloneqq - \int e^{\frac{{\rm c}_1(X)}{2}}\cdot \alpha^*\cdot \alpha'.
\end{equation}
\end{definition}
Here,
$(\alpha_0+\alpha_2+\alpha_4+\alpha_6+\alpha_8)^*\coloneqq \alpha_0-\alpha_2+\alpha_4-\alpha_6+\alpha_8$ and
$$e^{\frac{{\rm c}_1(X)}{2}}=e^{\frac{3h}{2}}=1+\frac{3}{2}h+\frac{9}{8}h^2+\frac{27}{48}h^3+\frac{81}{384}h^4.$$
\emph{Warning:} Unlike the Mukai pairing for K3 surfaces, the pairing (\ref{eqn:Mukaipcu}) is not symmetric.
\begin{definition}
The \emph{Mukai vector} of a coherent sheaf $E\in{\rm Coh}(X)$, or a complex $E\in\Db(X)$,  or simply a class $E\in K_{\rm top}(X)$
is defined as $$v(E)\coloneqq{\rm ch}(E)\cdot\sqrt{{\rm td}(X)}.$$
\end{definition}
One easily computes  $$\sqrt{{\rm td}(X)}=1+\frac{3}{4}h+\frac{11}{32}h^2+\frac{15}{128} h^3+\frac{121}{6144}h^4.$$
 Using the general fact $\sqrt{{\rm td}}^*=e^{-\frac{{\rm c}_1(X)}{2}}\cdot\sqrt{{\rm td}}$ and the Grothendieck--Riemann--Roch formula, one
 expresses the Euler--Poincar\'e pairing of two coherent sheaves as
 \begin{equation}\label{eqn:chiPoincub}
 \chi(E,E')=-(v(E).v(E')).
 \end{equation}
 Note that the left hand side is not symmetric, as $\omega_X$ is not trivial. This confirms the observation that (\ref{eqn:Mukaipcu}) is not symmetric.
 
 \begin{ex} For our purposes the following classes are of importance:
 $$w_0\coloneqq v(\ko_X)=\sqrt{{\rm td}(X)},~~Êw_1\coloneqq v(\ko_X(1))=e^h\cdot \sqrt{{\rm td}(X)},$$
 $$\text{ and }w_2\coloneqq v(\ko_X(2))=e^{2h}\cdot \sqrt{{\rm td}(X)}.$$
 In a sense to be made more precise, these classes are responsible for $(~.~)$ not being symmetric. Explicitly, they are
$$w_0=1+\frac{3}{4}h+\frac{11}{32}h^2+\frac{15}{128} h^3+\frac{121}{6144}h^4,~~w_1=1+\frac{7}{4}h+\frac{51}{32}h^2+\frac{385}{384} h^3+\frac{2921}{6144}h^4,$$
$$\text{ and } ~~w_2=1+\frac{11}{4}h+\frac{132}{32}h^2+\frac{1397}{384} h^3+\frac{16025}{6144}h^4.$$
 \end{ex}

In addition to the classes $w_0,w_1,w_2$, one also needs  the following ones
$$v(\lambda_1)\coloneqq 3+\frac{5}{4}h-\frac{7}{32} h^2-\frac{77}{384} h^3+\frac{41}{2048}h^4.$$
 $$v(\lambda_2)\coloneqq-3-\frac{1}{4}h+\frac{15}{32} h^2+\frac{1}{384} h^3-\frac{153}{2048}h^4.$$

\begin{remark}
The notation suggests that the $v(\lambda_i)$, $i=1,2$, are Mukai vectors of some natural (complexes of) sheaves. This is almost true, as we explain next.
Consider an arbitrary line $L\subset X$ and the two natural sheaves $\ko_L(i)$, $i=1,2$, on $X$. Their Mukai vectors
are $$u_i\coloneqq v(\ko_L(i))=
\begin{cases}\frac{1}{3}h^3+\frac{5}{12}h^4&\text{ if }i=1\\
\frac{1}{3}h^3+\frac{9}{12}h^4&\text{ if }i=2.
\end{cases}$$
Under the right orthogonal projection $H^*(X,\QQ)\to \{w_0,w_1,w_2\}^\perp$ they are mapped to $\lambda_i$.
Explicitly, 
\begin{equation}\label{eqn:imlambda}
v(\lambda_1)=u_1-w_1+4w_0\text{ and }v(\lambda_2)=u_2-w_2+4w_1-6w_0.
\end{equation}
Here, one uses $(u_i.u_j)=0$ for all $i,j$ and $$\begin{array}{rclll}
(w_i.w_j)&=&\chi(\ko_X(i),\ko_X(j))&=&\chi(X,\ko_X(j-i)),\\
(w_i.u_j)&=&\chi(\ko_X(i),\ko_L(j))&=&\chi(\PP^1,\ko_{\PP^1}(j-i)),\\
(u_i,w_j)&=&\chi(\ko_L(i),\ko_X(j))&=&\chi(\PP^1,\ko_{\PP^1}(i-j-3)),\\
\end{array}$$
\end{remark}

\begin{lem}\label{lem:symoncompl}
If $H^*(X,\QQ)$ is considered with the negative Mukai pairing, then 
 $$A_2\,\hookrightarrow H^*(X,\QQ),~\lambda_i\mapsto v(\lambda_i)$$ defines an isometric embedding. Furthermore,
\begin{enumerate}
\item[{\rm (i)}] $v(\lambda_1),v(\lambda_2)\in\{w_0,w_1,w_2\}^\perp$.
\item[{\rm (ii)}] $w_0,w_1,w_2,v(\lambda_1),v(\lambda_2)\in\QQ[h]$ are linearly independent.
\item[{\rm (iii)}] $\{w_0,w_1,w_2,v(\lambda_1),v(\lambda_2)\}^\perp=H^4(X,\QQ)_{\rm pr}={}^\perp\{w_0,w_1,w_2,v(\lambda_1),v(\lambda_2)\}$, on which the
Mukai pairing coincides with the intersection product (up to sign).
\item[{\rm (iv)}] The Mukai pairing $(~.~)$ is symmetric on the right orthogonal complement $$\{w_0,w_1,w_2\}^\perp\subset H^*(X,\QQ).$$
\end{enumerate}
\end{lem}

\begin{proof}
The first assertion can be verified by a computation or using (\ref{eqn:imlambda}). Similarly, (i) follows from the observation that $v(\lambda_i)$ is the orthogonal projection of $u_i$ and (ii) is again proven by a computation. Finally, (ii) implies (iii) and (iv) can be deduced from (iii).
\end{proof}

\begin{cor}
The lattices $A_2^\perp\cong\Gamma\cong H^4(X,\ZZ)_{\rm pr}\subset H^*(X,\QQ)$ and $A_2\cong \ZZ\,v(\lambda_1)\oplus\ZZ\, v(\lambda_2)\subset H^*(X,\QQ)$
are orthogonal with respect to the Mukai pairing (\ref{eqn:Mukaipcu}). The induced embedding of their direct sum $A_2^\perp\oplus A_2$
extends to 
\begin{equation}\label{eqn:tildeLH}
A_2^\perp\oplus A_2\subset\widetilde\Lambda\,\hookrightarrow H^*(X,\QQ).
\end{equation}
\vskip-0.7cm\qed
\end{cor}


A more conceptual understanding of these calculations is provided by the discussion in \cite{AT}.
In particular, cohomology with rational coefficients $H^*(X,\QQ)$ is replaced by integral
topological K-theory. Denote by $K_{\rm top}(X)$ the topological K-theory of all complex
vector bundles. Traditionally, the Chern character is used to identify $K_{\rm top}(X)\otimes\QQ$ with
$H^*(X,\QQ)=H^{2*}(X,\QQ)$. For our purposes the Mukai vector is better suited
$$v\colon K_{\rm top}(X)\,\hookrightarrow K_{\rm top}(X)\otimes\QQ\congpf H^*(X,\QQ).$$
Note that the torsion freeness of $K_{\rm top}(X)$ follows from the torsion freeness of $H^*(X,\ZZ)$ and  the Atiyah--Hirzebruch spectral sequence. Then $K_{\rm top}(X)$ is equipped with a non-degenerate but non-symmetric
linear form with values in $\QQ$. Due to (\ref{eqn:chiPoincub}), it takes values in $\ZZ$ on the image
of the highly non-injective map $K(X)\to K_{\rm top}(X)$. Clearly,
the classes $[\ko_X(i)]$, $i=0,1,2$, and $[\ko_L(i)]$, $i=1,2$, are all contained in the image.
We shall be interested in the right orthogonal complement of the former three classes and
introduce the notation:
$$K_{\rm top}'(X)\coloneqq \left\{\, [\ko_X], [\ko_X(1)],[\ko_X(2)]\,\right\}^\perp\subset K_{\rm top}(X).$$

\begin{prop}[Addington--Thomas]\label{prop:ATKtop}
The restriction of the Mukai pairing $(~.~)=-\chi(~,~)$ to $K_{\rm top}'(X)$ is symmetric and integral
Moreover, as abstract lattices 
$$\widetilde\Lambda\cong K_{\rm top}'(X).$$
\end{prop}

\begin{proof} 
Note that $v\colon K_{\rm top}'(X)\otimes\QQ\congpf\{w_0,w_1,w_2\}^\perp$. Hence, Lemma \ref{lem:symoncompl} implies
the first assertion. The original proof \cite{AT}  of the second assertion uses derived categories.
Here is a sketch of a more direct, purely topological argument. Consider the right orthogonal projection $p\colon K_{\rm top}(X)\twoheadrightarrow
K_{\rm top}'(X)$. It really is defined
over $\ZZ$, as $(w_i)^2=1$.  Analogously to (\ref{eqn:imlambda}), one has $p[\ko_L(1)]=[\ko_L(1)]-[\ko_X(1)]+4[\ko_X]$
and $p[\ko_L(2)]=[\ko_L(2)]-[\ko_X(2)]+4[\ko_X(1)]-6[\ko_X]$. Hence, $\lambda_i\mapsto p[\ko_L(i)]$ defines
an isometric embedding $A_2\,\hookrightarrow K_{\rm top}'(X)$.

First, $H^4(X,\ZZ)_{\rm pr}\subset H^*(X,\QQ)$ is contained in $v(K_{\rm top}'(X))$. Indeed, $H^4(X,\ZZ)_{\rm pr}$ is spanned by
 classes of all vanishing spheres and those lift to $K_{\rm top}(X)$. 
 After fixing an isometry $E\oplus U^{\oplus 2}\oplus A_2(-1)\cong A_2^\perp\cong\Gamma\cong H^4(X,\ZZ)_{\rm pr}\subset K_{\rm top}'(X)$, this yields an isometric
embedding 
$\Gamma\oplus A_2\,\hookrightarrow K_{\rm top}'(X)$ and allows one to view $\mu_1,\mu_2\in A_2(-1)$ as classes in
$K_{\rm top}'(X)$.\\
Second, one needs to show that the class $(1/3)(\mu_1-\mu_2-\lambda_1+\lambda_2)\in (A_2(-1)\oplus A_2)\otimes\QQ\subset K_{\rm top}(X)\otimes\QQ$ is integral, i.e.\ contained in $K_{\rm top}(X)$. This presumably can be achieved algebraically on some particular cubic fourfold.\footnote{This mysterious class would need to satisfy the two equations
$(\lambda_1.\alpha)=-1$ and $(\lambda_2.\alpha)=1$.} Hence, the embedding in step one extends to an isometric embedding $\widetilde \Lambda\,\hookrightarrow K_{\rm top}'(X)$ of finite index. The unimodularity of $\widetilde\Lambda$ then implies the second assertion.
\end{proof}

\subsection{}\label{sec:HodgeAssoK3C}
 We now endow the various lattices considered above with  natural Hodge structures. Let us first briefly recall the
well known theory for K3 surfaces, see \cite[Ch.\ 16]{HuyK3} for further details and references.

For any complex K3 surface $S$ its second cohomology $H^2(S,\ZZ)$, which as a lattice is isomorphic to $\Lambda$,
comes with a natural Hodge structure of weight two given by the $(2,0)$-part $H^{2,0}(S)$. The full Hodge structure is then determined by
additionally requiring  $H^{1,1}(S)\perp H^{2,0}(S)$ with respect to the intersection pairing.

The global Torelli theorem for complex K3 surfaces asserts that two K3 surfaces $S$ and $S'$ are isomorphic if and only
if there exists a Hodge isometry $H^2(S,\ZZ)\cong H^2(S',\ZZ)$, i.e.\ an isomorphism of integral Hodge structures 
that is compatible  with the intersection pairing:
$$S\cong S'\Leftrightarrow\exists\,  H^2(S,\ZZ)\cong H^2(S',\ZZ)\text{ Hodge isometry. }$$
Let $(S,L)$ be a polarized K3 surface. Then the primitive cohomology $H^2(S,\ZZ)_{L\text{\rm -pr}}\subset H^2(S,\ZZ)$ is endowed
with the induced structure. Its $(2,0)$-part is again $H^{2,0}(S)$ and its $(1,1)$-part is the primitive
part of $H^{1,1}(S)$, i.e.\ the kernel of $(L.~)\colon H^{1,1}(S)\to\CC$. The polarized version of the global Torelli theorem is the
the statement that two polarized K3 surfaces $(S,L)$ and $(S',L')$ are isomorphic if and only if there exists a
Hodge isometry $H^2(S,\ZZ)\cong H^2(S',\ZZ)$ inducing $H^2(S,\ZZ)_{L\text{\rm -pr}}\cong H^2(S',\ZZ)_{L\text{\rm -pr}}$:
$$(S,L)\cong(S',L)\Leftrightarrow\exists\, H^2(S;\ZZ)\cong H^2(S',\ZZ),\, L\mapsto L', \text{ Hodge isometry }\\$$
The result will be stated again in moduli theoretic terms in Theorem \ref{thm:GTK3}.
\smallskip

\emph{Warning:} A Hodge isometry $H^2(S,\ZZ)_{L\text{\rm -pr}}\cong H^2(S',\ZZ)_{L\text{\rm -pr}}$ does not necessarily extend
to a Hodge isometry between the full cohomology. Hence, in general, the existence of a Hodge isometry between the primitive
Hodge structures of two polarized K3 surfaces does not imply that $(S,L)$ and $(S',L')$ are isomorphic. In fact, even
the unpolarized K3 surfaces $S$ and $S'$ may be non-isomorphic.

\smallskip

Next comes the Mukai Hodge structure $\widetilde H(S,\ZZ)$.
The underlying lattice is $H^*(S,\ZZ)$ with the sign change in $U_4=(H^0\oplus H^4)(S,\ZZ)$.
The Hodge structure of weight two is again given by the $(2,0)$-part being $\widetilde H^{2,0}(S)\coloneqq H^{2,0}(S)$ and the condition
that $\widetilde H^{1,1}(S)\perp H^{2,0}(S)$ with respect to the Mukai pairing. In particular, $U\cong U_4=(H^0\oplus H^4)(S,\ZZ)$ is contained
in $\widetilde H^{1,1}(S,\ZZ)$. The derived global Torelli theorem is the statement that for two projective K3 surfaces $S$ and $S'$ there
exists an exact, $\CC$-linear equivalence $\Db(S)\cong \Db(S')$ between their bounded derived categories of coherent sheaves if and only
of there exists a Hodge isometry $\widetilde H(S,\ZZ)\cong \widetilde H(S',\ZZ)$:
$$\Db(S)\cong \Db(S')\Leftrightarrow \exists\, \widetilde H(S,\ZZ)\cong\widetilde H(S',\ZZ)\text{ Hodge isometry.}$$

A twisted K3 surface $(S,\alpha)$ consists of a K3 surface $S$ together with a Brauer class $\alpha\in {\rm Br}(S)\cong H^2(S,\ko_S^\ast)$ (we work in the 
analytic topology). Choosing a lift $B\in H^2(S,\QQ)$ of $\alpha$ under the natural morphism
$H^2(S,\QQ)\to {\rm Br}(S)$ induced by the exponential sequence allows one to introduce
a natural Hodge structure $\widetilde H(S,\alpha,\ZZ)$ of weight two associated with $(S,\alpha)$.
As a lattice, this is just $\widetilde H(S,\ZZ)$, but the $(2,0)$-part is now given by
$\widetilde H^{2,0}(S,\alpha)\coloneqq \CC\,(\sigma+\sigma\wedge B)$, where
$0\ne\sigma\in H^{2,0}(S)$. This defines a Hodge structure by requiring, as before, that
$\widetilde H^{1,1}(S,\alpha)\perp \widetilde H^{2,0}(S,\alpha)$ with respect to the Mukai pairing.
Although the definition depends on the choice of $B$, the Hodge structures
induced by two different lifts $B$ and $B'$ of the same Brauer class $\alpha$ are Hodge isometric albeit not canonically, see \cite{HuyStel}.

The twisted version of the derived global Torelli theorem is the statement that the bounded derived categories
of twisted coherent sheaves on $(S,\alpha)$ and $(S',\alpha')$ are equivalent if and only
if there exists a Hodge isometry $\widetilde H(S,\alpha,\ZZ)\cong \widetilde H(S',\alpha',\ZZ)$ preserving the natural orientation
of the four positive directions, cf.\ \cite[Ch.\ 16.4]{HuyK3} and \cite{Rein}:
 $$\Db(S,\alpha)\cong \Db(S',\alpha')\Leftrightarrow \exists\,\widetilde H(S,\alpha,\ZZ)\cong\widetilde H(S',\alpha',\ZZ)\text{ oriented Hodge isometry.}$$
Next consider $H^4(X,\ZZ)$ and  $H^4(X,\ZZ)_{\rm pr}$  of a smooth cubic fourfold $X$. These are Hodge structures
of weight four determined by the one-dimensional $H^{3,1}(X)$ and the condition that $H^{3,1}(X)\perp H^{2,2}(X)$ with respect
to the intersection product. 

The global Torelli theorem for smooth cubic fourfolds, which we will state again as Theorem \ref{thm:GT3} in moduli theoretic terms,
is the statement that two smooth cubic fourfolds $X$ and $X'$ are isomorphic (as abstract complex varieties) if and only
if there exists a Hodge isometry $H^4(X,\ZZ)_{\rm pr}\cong H^4(X',\ZZ)_{\rm pr}$:
$$X\cong X'\Leftrightarrow \exists\, H^4(X,\ZZ)_{\rm pr}\cong H^4(X',\ZZ)_{\rm pr}\text{ Hodge isometry}.$$
 Note that any such Hodge isometry can be extended
to a Hodge isometry $H^4(X,\ZZ)\cong H^4(X',\ZZ)$ that maps $h_X^2$ to $\pm h_{X'}^2$. The situation here is easier compared to the case
of polarized K3 surfaces as the discriminant of $H^4(X,\ZZ)_{\rm pr}$ is just $\ZZ/3\ZZ$.\footnote{We will
encounter yet another Torelli theorem in Section
\ref{sec:GTHK4}.}

\smallskip

To relate $H^4(X,\ZZ)$ of a cubic fourfolds  to K3 surfaces one has to change the sign of the intersection product, so that
as abstract lattices $H^4(X,\ZZ)\cong\bar\Gamma$ and $H^4(X,\ZZ)_{\rm pr}\cong\Gamma$  (with an implicit sign change), and Tate shift the Hodge structure to obtain
$H^4(X,\ZZ)(1)$ and $H^4(X,\ZZ)_{\rm pr}(1)$, which are now Hodge structures of weight two.

\begin{definition} 
The integral Hodge structure $\widetilde H(X,\ZZ)$ of K3 type associated with a smooth cubic fourfold $X$ is
 the lattice $$\widetilde H(X,\ZZ)\coloneqq K_{\rm top}'(X)$$ 
with the Hodge structure of weight two given by $\widetilde H^{2,0}(X)\coloneqq v^{-1}(H^{3,1}(X))$ and the 
requirement that $\widetilde H^{1,1}(X)$ and $\widetilde H^{2,0}(X)$ are orthogonal with respect to the Mukai pairing on
$K_{\rm top}(X)$.
\end{definition}

The Mukai vector $K_{\rm top}(X)\otimes\QQ\congpf H^*(X,\QQ)$ induces an isometry
$$\widetilde H(X,\ZZ)= K'_{\rm top}(X)\cong\widetilde\Lambda\subset H^*(X,\QQ)$$
with $\widetilde\Lambda\subset H^*(X,\QQ)$ provided by (\ref{eqn:tildeLH}).
Observe that there is a natural isometric inclusion of Hodge structures
$$H^4(X,\ZZ)_{\rm pr}(1)\subset \widetilde H(X,\ZZ).$$
Moreover, the sublattice $A_2$ is algebraic, i.e.\ $A_2\subset\widetilde H^{1,1}(X,\ZZ)$, and its orthogonal
Hodge structure is $A_2^\perp\cong H^4(X,\ZZ)_{\rm pr}(1)$. Also note that according
to Remark \ref{rem:l1perp} $\lambda_1^\perp\subset \widetilde H(X,\ZZ)$
is a sub Hodge structure with underlying lattice  isomorphic to $\Lambda\oplus\ZZ(-2)$.

\begin{remark}\label{rem:KuznH}
Once the Kuznetsov category $\ka_X\subset \Db(X)$ is  introduced, one also writes
$\widetilde H(\ka_X,\ZZ)=\widetilde H(X,\ZZ)$. The notation $\widetilde H(X,\ZZ)$ is analogous to the notation $\widetilde H(S,\ZZ)$ for
K3 surfaces and the Hodge structure plays a similar role. In fact, as a consequence of the above discussion we know that as lattices
$\widetilde H(X,\ZZ)\cong \widetilde H(S,\ZZ)$ and the analogy goes further:
For a K3 surface, the algebraic part naturally contains a  hyperbolic plane:
$$U\cong (H^0\oplus H^4)(S,\ZZ)\,\hookrightarrow \widetilde H^{1,1}(S,\ZZ).$$
Similarly, for a smooth cubic fourfold the algebraic part naturally contains a copy of $A_2$:
$$v\colon A_2\cong \ZZ\, p[\ko_L(1)]\oplus\ZZ\, p[\ko_L(2)] \,\hookrightarrow \widetilde H^{1,1}(X,\ZZ).$$
Here, $p\colon K_{\rm top}(X)\to K'_{\rm top}(X)$ is the projection as in the proof of Proposition \ref{prop:ATKtop}, so the composition maps $\lambda\mapsto p[\ko_L(i)]\mapsto v(\lambda_i)$.
Their respective orthogonal complements are
$$H^2(S,\ZZ)=U^\perp\,\hookrightarrow \widetilde H(S,\ZZ)~\text{ and }~ H^4(X,\ZZ)_{\rm pr}(1)=A_2^\perp\,\hookrightarrow \widetilde H(X,\ZZ),$$
in terms of which the global Torelli theorem is formulated in both instances.
Also, $e_4-f_4=(1,0,-1)\in\widetilde H^{1,1}(S,\ZZ)$ and $v(\lambda_1)\in\widetilde H^{1,1}(X,\ZZ)$ are both algebraic classes
satisfying $(e_4-f_4)^2=2=(v(\lambda_1))^2$. Their orthogonal complements are isometric.
\end{remark}

\begin{definition}\label{def:assoK3C} Let $(S,L)$ be a polarized K3 surface and $X$ a smooth cubic fourfold.
\begin{enumerate}
\item[{\rm (i)}] We say $(S,L)$ and $X$ are \emph{associated}, $(S,L)\sim X$,
if there exists an isometric embedding of Hodge structures
\begin{equation}\label{eqn:HassettAsso}
H^2(S,\ZZ)_{L\text{\rm -pr}}\,\hookrightarrow H^4(X,\ZZ)_{\rm pr}(1).
\end{equation}
\item[{\rm (ii)}]  We say $S$ and $X$ are \emph{associated}, $S\sim X$,
if there exists a Hodge isometry
$$\widetilde H(S,\ZZ)\cong \widetilde H(X,\ZZ).$$
\item[{\rm (iii)}] For $\alpha\in{\rm Br}(S)$ we say that the twisted K3 surface $(S,\alpha)$ and $X$ are
 \emph{associated}, $(S,\alpha)\sim X$,
if there exists a Hodge isometry
$$\widetilde H(S,\alpha,\ZZ)\cong \widetilde H(X,\ZZ).$$
\end{enumerate}
\end{definition}
First observe the immediate implication:
$$(S,L)\sim X~\Rightarrow S\sim X.$$
Indeed, any isometric embedding (\ref{eqn:HassettAsso}) can be extended
to an isometry $\widetilde H(S,\ZZ)\cong \widetilde H(X,\ZZ)$. This follows from the existence of the hyperbolic plane $U\subset H^2(S,\ZZ)_{L\text{\rm -pr}}^\perp$, cf.\ \cite[Rem.\ 14.1.13]{HuyK3}.

 As an aside, observe that a K3 surface $S$ that is associated with a cubic fourfold
in any sense is necessarily projective. Indeed, if for example $S\sim X$, then $\widetilde H^{1,1}(S,\ZZ)\cong \widetilde H^{1,1}(X,\ZZ)$ contains
the positive plane $A_2$ and, therefore, $H^{1,1}(S,\ZZ)$ contains at least one class of positive square.

The key to link $S\sim X$, $(S,L)$, and $(S,\alpha)\sim X$ to the properties $(\ast\ast)$ and $(\ast\ast')$ is the following result
in \cite{AT} generalized to the twisted case in \cite{HuyComp}.
\begin{prop}[Addington--Thomas, Huybrechts]\label{prop:UUn}
Assume $X$ is a smooth cubic fourfold.
\begin{enumerate}
\item[{\rm (i)}] There exists a K3 surface $S$ with $S\sim X$ if and only if there exists a (primitive)
 embedding $U\,\hookrightarrow \widetilde H^{1,1}(X,\ZZ)$.
\item[{\rm (ii)}] There exists a twisted K3 surface $(S,\alpha)$ with $(S,\alpha)\sim X$ if and only if there exists an
embedding $U(n)\,\hookrightarrow \widetilde H^{1,1}(X,\ZZ)$ for some $n\ne 0$.
\end{enumerate}
\end{prop}

\begin{proof}
Any Hodge isometry $\widetilde H(S,\ZZ)\cong\widetilde H(X,\ZZ)$ yields a hyperbolic plane
$U\cong (H^0\oplus H^4)(S,\ZZ)\subset\widetilde H^{1,1}(S,\ZZ)\cong \widetilde H^{1,1}(X,\ZZ)$. Conversely, if
$U\subset \widetilde H^{1,1}(X,\ZZ)\subset\widetilde H(X,\ZZ)$, then as a lattice $U^\perp\cong\Lambda$. Moreover,
the Hodge structure of $\widetilde H(X,\ZZ)$ induces
a Hodge structure on $U^\perp\cong\Lambda$ which due to the surjectivity of the period map \cite[Thm.\ 7.4.1]{HuyK3} is Hodge isometric
to $H^2(S,\ZZ)$ for some K3 surface $S$. However, as before,
$U^\perp\cong H^2(S,\ZZ)$ extends to $\widetilde H(X,\ZZ)\cong \widetilde H(S,\ZZ)$.
This proves (i). 

For (ii), again one direction is easy, as $\widetilde H^{1,1}(S,\ZZ)$ contains the B-field shift
of $(H^0\oplus H^4)(S,\ZZ)$, cf.\ \cite[Ch.\ 14]{HuyK3}.
 More precisely,  $\widetilde H^{1,1}(S,\alpha,\ZZ)=
(\exp(B)\,\widetilde H^{1,1}(S,\QQ))\cap \widetilde H(S,\ZZ)$, which clearly contains the lattice
$(\langle1,B,B^2/2\rangle\cap \widetilde H(S,\ZZ))\oplus H^4(S,\ZZ)\cong U(n)$, where $n$ is minimal with $n\,(1,B,B^2)\in \widetilde H(S,\ZZ)$.
The other direction needs a surjectivity statement for twisted K3 surfaces which is an easy
consequence of the surjectivity of the untwisted period map.
\end{proof}

\begin{prop}
Assume a smooth cubic fourfold $X$ is associated with some K3 surface $S$, so $S\sim X$. Then there exists a polarized
K3 surface $(S',L')\sim X$:
$$\widetilde H(S,\ZZ)\cong\widetilde H(X,\ZZ)~~\Rightarrow~~H^2(S',\ZZ)_{L'\text{\rm -pr}}\hookrightarrow H^4(X,\ZZ)_{\rm pr}.$$
\end{prop}

\begin{proof} 
Assume $S\sim X$. Then there exists a Hodge isometry $\widetilde H(S,\ZZ)\cong\widetilde H(X,\ZZ)$.
On  the left hand side, one finds $U\cong (H^0\oplus H^4)(S,\ZZ)\subset \widetilde H^{1,1}(S,\ZZ)$ and, on the right hand side, $A_2\subset\widetilde H^{1,1}(X,\ZZ)$.
Consider the saturation of the  sum of both as a lattice $\overline{U+A_2}\subset \widetilde H^{1,1}(S,\ZZ)$. According to
 Lemma \ref{lem:UA2AT}, there exists another hyperbolic plane $U'\subset \overline{U+A_2}$ with $\rk(U'+A_2)=3$.
 Using the surjectivity of the period map, one finds another K3 surface $S'$ and a Hodge isometry
 \begin{equation}\label{eqn:newHI}
 \widetilde H(S',\ZZ)\cong
 \widetilde H(S,\ZZ)\cong \widetilde H(X,\ZZ)
 \end{equation} inducing $H^2(S',\ZZ)\cong U'^\perp$. But then $H^2(S',\ZZ)\cap A_2^\perp\subset H^2(S',\ZZ)$ is of corank one
 and we can assume it to be of the form $H^2(S',\ZZ)_{L'\text{\rm-pr}}$. However, being contained in $A_2^\perp$ implies that under (\ref{eqn:newHI})
 $H^2(S',\ZZ)_{L'\text{\rm -pr}}$ embeds into $H^4(X,\ZZ)_{\rm pr}(1)$, which ensures $(S',L')\sim X$.
\end{proof}

\begin{cor}
A smooth cubic fourfold $X$ is associated with some polarized K3 surface, $(S,L)\sim X$, if and only if there exists an 
isometric  embedding $U\,\hookrightarrow \widetilde H^{1,1}(X,\ZZ)$.\qed
\end{cor}


\section{Period domains and moduli spaces}
The comparison of the Hodge theory of K3 surfaces and cubic fourfolds is now considered in families. Via period maps, this leads to
an algebraic correspondence between the moduli space of polarized K3 surfaces of certain degrees and the moduli space of cubic fourfolds.
The approach has been initiated by Hassett \cite{HassComp} and has turned out to be very valuable indeed.

\subsection{} Here is a very brief reminder on some results, mostly due to Borel and Baily--Borel, on arithmetic quotients of orthogonal type.
Let $(N, (~.~))$ be a lattice of signature $(2,n_-)$ and set $V\coloneqq N\otimes\RR$. Then the  period domain $D_N$
associated with $N$ is the Grass\-mannian of positive, oriented planes $W\subset V$, which alternatively can be described
as
\begin{eqnarray*}
D_N&\cong&\{~Êx\mid (x)^2=0,~(x.\bar x)>0~Ê\}\subset\PP(N\otimes \CC)\\
&\cong&\OO(2,n_-)/(\OO(2)\times\OO(n_-)).
\end{eqnarray*}
By definition, the period domain $D_N$ associated with $N$ has the structure of a complex manifold. This is turned into an algebraic
statement by the following fundamental result \cite{BB}. It uses the fact that under the assumption on the signature of $N$ the orthogonal
group $\OO(N)$ acts properly discontinuously on $D_N$.

\begin{thm}[Baily--Borel]\label{thm:BB}
Assume $G\subset\OO(N)$ is a torsion free subgroup of finite index. Then the quotient
$$G\setminus D_N$$
has the structure of a smooth, quasi-projective complex variety.
\end{thm}

As $G$ acts properly discontinuously as well, the stabilizers are finite and hence trivial. This already proves the
smoothness of the quotient $G\setminus D_N$. The difficult part of the theorem is to find a Zariski open embedding into a complex projective
variety. 

Finite index subgroups $G\subset\OO(N)$ with torsion are relevant, too. In this situation, one uses Minkowski's theorem stating
that the map $\pi_p\colon{\rm Gl}(n,\ZZ)\to {\rm Gl}(n,\FF_p)$, $p\geq 3$, is injective on finite subgroups or, equivalently, that
its kernel is torsion free. Hence, for every finite index subgroup $G\subset\OO(N)$ there exists
a normal and torsion free subgroup $G_0\coloneqq G\cap{\rm Ker}(\pi_p)\subset G$ of finite index.

\begin{cor}
Assume $G\subset \OO(N)$ is a subgroup of finite index. Then the quotient $G\setminus D_N$
has the structure of a normal, quasi-projective complex variety with finite quotient singularities.\qed
\end{cor}

We remark that not only  these arithmetic quotients, but also holomorphic maps into them are algebraic. This is the
following remarkable GAGA style result, see \cite{Borel}.

\begin{thm}[Borel]\label{thm:BorelJDG}
Assume $G\subset\OO(N)$ is a torsion free subgroup of finite index. Then any holomorphic map
$\varphi\colon Z\to G\setminus D_N$  from a complex variety $Z$ is regular.
\end{thm}

\begin{remark}\label{rem:BBnontf}
Often, the result is applied to holomorphic maps to singular quotients $G\setminus D_N$, i.e.\ in situations when $G$
is not necessarily torsion free. This is covered by the above only when $Z\to G\setminus D_N$ is induced by
a holomorphic map $Z'\to G_0\setminus D_N$, where $Z'\to Z$ is a finite quotient and $G_0\subset G$ is a 
normal, torsion free subgroup of finite index.
\end{remark}

\subsection{} 
We shall be interested in (at least) three different types of period domains: For polarized K3 surfaces and
for (special) smooth cubic fourfolds. These are the period domains associated with the lattices
$\Gamma$, $\Gamma_d$, and $\Lambda_d$:
$$D\subset \PP(\Gamma\otimes\CC),~ D_d\subset\PP(\Gamma_d\otimes\CC),\text{ and } Q_d\subset\PP(\Lambda_d\otimes \CC).$$

These period domains are endowed with the natural action of the corresponding orthogonal groups $\OO(\Gamma)$,
$\OO(\Gamma_d)$, and $\OO(\Lambda_d)$ and we will be interested in the following quotients by
distinguished finite index subgroups of those:
$$\kc\coloneqq \tilde\OO(\Gamma)\setminus D=\OO(\Gamma)\setminus D,$$
$$\tilde{\!\kc}_d\coloneqq \tilde\OO(\Gamma,K_d)\setminus D_d, ~~~
\, \cctwotilde_d\coloneqq \tilde\OO(\Gamma,v_d)\setminus D_d, \text{ and }$$
$$ \km_d\coloneqqÊ\tilde\OO(\Lambda_d)\setminus Q_d.$$
For the first equality note that $\tilde\OO(\Gamma)\subset\OO(\Gamma)$ is of index two, but $-{\rm id}\in
\OO(\Gamma)\setminus  \tilde\OO(\Gamma)$ acts trivially on $D$. The subgroup $Ê\tilde\OO(\Lambda_d)\subset\OO(\Lambda_d)$
is defined analogously to (\ref{eqn:tildeOO}).

Due to Theorem \ref{thm:BB} and \ref{thm:BorelJDG}, see also Remark \ref{rem:BBnontf}, the induced maps
$\cctwotilde_d\twoheadrightarrow\,\tilde{\!\kc}_d\to\kc$
are regular morphisms between normal quasi-projective varieties.
The image in $\kc$ shall be denoted by $\kc_d$, so that
$$\cctwotilde_d\twoheadrightarrow\,\tilde{\!\kc}_d\twoheadrightarrow\kc_d\subset\kc.$$ The condition $(\ast)$
will in the sequel be interpreted as the condition that $\kc_d\ne\emptyset$.

\begin{cor}[Hassett] Assume $d$ satisfies $(\ast)$.
The naturally induced maps
$$\cctwotilde_d\twoheadrightarrow\,\tilde{\!\kc}_d\twoheadrightarrow\kc_d$$
are surjective, finite, and algebraic. 

Furthermore, $\,\tilde{\!\kc}_d\twoheadrightarrow\kc_d$ is the normalization of $\kc_d$
and  $\cctwotilde_d\to\,\tilde{\!\kc}_d$ is a finite morphism   between normal varieties, which is an isomorphism
if $d\equiv2\,(6)$ and of degree two if $d\equiv0\,(6)$.
\end{cor}

\begin{proof}
Clearly, if $d$ satisfies $(\ast)_2$, then $\tilde\OO(\Gamma,K_d)=\tilde\OO(\Gamma,v_d)$  by Lemma \ref{lem:KDOO}
and, therefore, $\cctwotilde_d\cong\,\tilde{\!\kc}_d$. Otherwise, $\cctwotilde_d\twoheadrightarrow \,\tilde{\!\kc}_d $ is the quotient by
the involution $g\in \tilde\OO(\Gamma)$ defined by $g={\rm id}$ on $E\oplus U_2\oplus {\rm I}_{0,3}$ and $g=-{\rm id}$
on $U_1$, which indeed acts non-trivially on $\cctwotilde_d$. 

To prove that $\,\tilde{\!\kc}_d\twoheadrightarrow\kc_d$
is quasi-finite, use that $\,\tilde{\!\kc}_d\to\kc$ is algebraic with discrete and hence finite fibres.
For a very general $x\in D_d$ such that there does not exist any proper primitive sublattice $N\subset \Gamma_d$
with $x\in N\otimes\CC$, any $g\in \tilde\OO(\Gamma)$ with $g(x)=x$ also satisfies $g(\Gamma_d)=\Gamma_d$ and, therefore,
$g(K_d)=K_d$, i.e.\ $g\in \tilde\OO(\Gamma,K_d)$. This proves that $\,\tilde{\!\kc}_d\to\kc$ is generically injective.
Thus, once $\,\tilde{\!\kc}_d\to\kc$ is shown to be finite, and not only quasi-finite, it is the normalization of its image $\kc_d$.
We refer to \cite{Brakkee,HassComp} for more details on this point.
\end{proof}

\begin{remark}
Note that while the fibre of $\cctwotilde_d\to\tilde{\!\kc}_d$ consists of 
at most two points, the fibres of $\,\tilde{\!\kc}_d\to\kc_d$ may contain more points, depending on
the singularity type of the points in $\kc_d$. For fixed $d$, the cardinality of the fibres is bounded. However,
it is unbounded when $d$ is allowed to grow.
\end{remark}

Lemma \ref{lem:OOLamL} immediately yields the following result which eventually leads to the mysterious relation between
K3 surfaces and cubic fourfolds.

\begin{cor}\label{cor:piceps}
Assume $d$ satisfies $(\ast\ast)$. We choose an isomorphism $\varepsilon\colon\Gamma_d\congpf\Lambda_d$.
\begin{enumerate}
\item[{\rm (i)}] If $d$ satisfies $(\ast)_0$, then $\varepsilon$ naturally induces an isomorphism $\km_d\cong\cctwotilde_d$.
Therefore, $\km_d$ comes with a finite morphism
onto $\kc_d$ generically of degree two:
$$\Phi_\varepsilon\colon\xymatrix{\km_d\cong\cctwotilde_d\ar[r]^-{2:1}&\,\tilde{\!\kc}_d\ar[r]^-{{\rm norm}}&\kc_d\subset\kc.}$$
\item[{\rm(ii)}]
If  $d$ satisfies $(\ast)_2$, then $\varepsilon$ naturally induces  an isomorphism
$\km_d\cong\cctwotilde_d\cong\,\tilde{\!\kc}_d$. Therefore, $\km_d$ can be seen as the normalization of $\kc_d\subset\kc$:
$$\xymatrix{\Phi_\varepsilon\colon\km_d\cong\cctwotilde_d\cong\,\tilde{\!\kc}_d\ar[r]^-{{\rm norm}}&\kc_d\subset\kc.}$$\vskip-0.7cm\qed
\end{enumerate}
\end{cor}

\begin{remark}\label{rem:finitequotMd}
As indicated by the notation, the morphism $\Phi_\varepsilon\colon\km_d\twoheadrightarrow \kc_d\subset\kc$,
which will be seen to link polarized K3 surfaces $(S,L)$ of degree $d$ with special cubic fourfolds $X$,
depends on the choice of $\varepsilon\colon\Gamma_d\congpf\Lambda_d$. There is no distinguished choice 
for $\varepsilon$ and, therefore, one should not expect to find a distinguished morphism $\km_d\to\kc_d$ 
that can be described by a geometric procedure associating a cubic fourfold $X$ to a polarized K3 surface $(S,L)$.\footnote{I wish to thank E.\ Brakkee and P.\ Magni for discussions concerning this point.}

To avoid any dependance on $\varepsilon$, one could think of defining a morphism from the finite quotient 
$$\pi_d\colon\km_d=  \tilde\OO(\Lambda_d)\setminus Q_d\twoheadrightarrow \bar\km_d\coloneqq\OO(\Lambda_d)\setminus Q_d$$ to some meaningful quotient of $\kc$. But, as the degree of $\pi_d$ grows with $d$,  there definitely is no reasonable
quotient of $\kc$ that would receive all of them. However, it seems plausible that a quotient  $\kc_d\to\bar{\,\kc_d}$ can be constructed
that allows for a morphism $\bar\km_d\twoheadrightarrow\bar{\,\kc_d}$. The derived point of view to be explained later will shed
more light on this.
\end{remark}

\subsection{} 

We start by recalling the central theorem in the theory of K3 surfaces: the global Torelli theorem. In the situation at hand, it
is due to Pjatecki{\u\i}-{\v{S}}apiro and {\v{S}}afarevi{\v{c}}, see \cite{HuyK3} for details, generalizations, and references.

Consider the coarse moduli space $M_d$ of polarized K3 surfaces $(S,L)$ with $(L)^2=d$, which can be constructed
as a quasi-projective variety either by (not quite) standard GIT methods, by using the theorem below, or as a Deligne--Mumford
stack. 

The period map associates with any $[(S,L)]\in M_d$ a point in $\km_d$. For this, choose an isometry
$H^2(S,\ZZ)\cong \Lambda$, called a marking, that maps ${\rm c}_1(L)$ to $\ell=e_2+(d/2)f_2$ and, therefore, induces
an isometry $H^2(S,\ZZ)_{L\text{\rm -pr}}\cong\Lambda_d$. Then the $(2,0)$-part $H^{2,0}(S)\subset H^2(S,\CC)\cong\Lambda\otimes\CC$ defines
a point  in the period domain $Q_d$. The image point in the quotient $\tilde\OO(\Lambda_d)\setminus Q_d$ is then independent
of the choice of any marking. This defines the period map $\kp\colon M_d\to \km_d$ which Hodge theory reveals to be holomorphic.
Note that both spaces, $M_d$ and $\km_d$, are quasi-projective varieties with quotient singularities.

\begin{thm}[Pjatecki{\u\i}-{\v{S}}apiro and {\v{S}}afarevi{\v{c}}]\label{thm:GTK3}
The period map is an algebraic, open embedding
\begin{equation}\label{thm:GTK3}
\kp\colon M_d\,\hookrightarrow \km_d =Ê\tilde\OO(\Lambda_d)\setminus Q_d.
\end{equation}
\end{thm}

\begin{remark}
Coming back to Remark \ref{rem:finitequotMd},
one might wonder how the image of $M_d$ under the finite quotient $\pi_d\colon\km_d\to\bar\km_d$, 
can be interpreted geometrically in terms of the polarized K3 surfaces $(S,L)$ parametrized by $M_d$. There is no completely satisfactory
answer to this, i.e.\ the image $\pi_d(M_d)$ is not known (and should probably not expected)
to be the coarse moduli space of a nice geometric moduli functor. The best one
can say is that for  $(S,L)\in M_d$ with $\rho(S)=1$, the fibre $\pi_d^{-1}(\pi_d(S,L))$ can be viewed as the set
of all Fourier--Mukai partners of $S$, which come with a unique polarization, cf.\ \cite{HulPlo,HuyFinite}.
\end{remark}

To understand the complement of the open embedding (\ref{thm:GTK3}), note first that
any $x\in Q_d$ is the period of some K3 surface $S$. This surface then comes with a natural line bundle $L$
(up to the action of the Weyl group) corresponding to $\ell=e_2+(d/2)f_2\in\Lambda$. Furthermore, $L$ is ample
(again, possibly after applying the Weyl group action)
if and only if there exists no $\delta\in\Lambda_d$ with $(\delta)^2=-2$ orthogonal to $x$, i.e.\
$x\in Q_d\setminus\bigcup \delta^\perp$ with $\delta\in\Delta_d\coloneqq\Delta(\Lambda_d)$, the set of all $(-2)$-classes in $\Lambda_d$.
Hence, the complement of $M_d\subset\km_d$ can be described as the quotient
\begin{equation}\label{eqn:compl}
\xymatrix{\tilde\OO(\Lambda_d)\setminus \bigcup\delta^\perp\subset\km_d.}
\end{equation}
Note that $\tilde\OO(\Lambda_d)$ acts on $\Delta_d$ and that  the quotient (\ref{eqn:compl}) really is a finite union.
In fact, it consists of at most two components due to the following result.\footnote{Thanks to O.\ Debarre for pointing this out to me.}

\begin{prop}\label{prop:imageperiodK3} The complement $\km_d\setminus M_d$ consists of either one or two irreducible
Noether--Lefschetz divisors depending on $d$:
\begin{enumerate}
\item[{\rm (i)}] If $d/2\not\equiv 1\, (4)$, then the complement (\ref{eqn:compl}) of $M_d\subset \km_d$ is irreducible.
\item[{\rm (ii)}] If $d/2\equiv 1\, (4)$, then the complement (\ref{eqn:compl}) of $M_d\subset \km_d$ has  of two irreducible components.
\end{enumerate}
\end{prop}

\begin{proof}
This is again an application of Eichler's criterion, see the proof of Proposition \ref{prop:HassettEichler}. For $\delta\in\Lambda_d$ with
$(\delta)^2=-2$, one has $(\delta.\Lambda_d)=n\,\ZZ$ with $n=1$ or $n=2$. In the first case, the residue class $(1/n)\,\bar\delta\in A_{\Lambda_d}\cong\ZZ/d\ZZ$ is trivial. In the second case,   $(1/2)\,\bar\delta\equiv0$ or $\equiv d/2\,(d)$ in $\ZZ/d\ZZ$.
However, the second case is only possible if $d/2\equiv1\, (4)$. Indeed, write $\delta=\delta'+\delta''\in U_2^\perp\oplus U_2$ with
$\delta''\in\ell^\perp\cap U_2=\ZZ\, (e_2-(d/2)f_2)$. Then $(1/2)\,\delta'+(1/2)\,\delta''+(m/2)\,\ell \in \Lambda$ for some $m\in \ZZ$.
Hence, $(1/2)\,\delta'\in\Lambda$ and, therefore, $-2=(\delta)^2\equiv (\delta'')^2\,(8)$. Combine this with
$(1/2)\,\delta''+(m/2)\,\ell\in U_2$, which implies $(\delta'')^2\equiv -m^2d\,(8)$.
\end{proof}
To be more explicit, one can write
$$M_d=\begin{cases}\km_d\setminus \delta_0^\perp&\text{if }Ê\frac{d}{2}\not\equiv 1\, (4)\\
\km_d\setminus(\delta_0^\perp\cup \delta_1^\perp)&\text{if }Ê\frac{d}{2}\equiv 1\, (4),\end{cases}$$
where $\delta_0,\delta_1$ are chosen explicitly as $\delta_0=e_1-f_1$ and $\delta_1=2e_1+\frac{d/2-1}{2}f_1+e_2-(d/2)\, f_2$.

\subsection{} 
We now switch to the cubic side. The moduli space $M$ of smooth cubic fourfolds can be constructed by means of standard
GIT methods as the quotient
$$M=|\ko_{\PP^5}(3)|_{\rm sm}/\!/{\rm PGl}(6).$$ As in the case of K3 surfaces, mapping a smooth cubic fourfold
$X$ to its period $H^{3,1}(X)\subset H^4(X,\CC)_{\rm pr}\cong \Gamma\otimes\CC$, which is a point
in the period domain $D\subset\PP(\Gamma\otimes\CC)$, defines a holomorphic map $\kp\colon M\to \kc$.
In analogy to the situation for K3 surfaces, the following global Torelli theorem has been proven \cite{VoisinGT,VoisinGTErr,LooPeriod,CharlesGT,HuyRen}.

\begin{thm}[Voisin, Looijenga,...,Charles, Huybrechts--Rennemo,...]\label{thm:GT3}
The period map is an algebraic, open embedding
$$\kp\colon M\,\hookrightarrow \kc=\OO(\Gamma)\setminus D.$$
\end{thm}

This central result is complemented by a result of Laza and Looijenga,  which can be seen as an analogue of Proposition \ref{prop:imageperiodK3}, see \cite{LazaCubics,LooPeriod}. First note that for $d=2$ and $d=6$ the lattice $K_d$ is given by the matrices $\left(\begin{matrix}-3&1\\1&-1\end{matrix}\right)$
and $\left(\begin{matrix}-3&0\\0&-2\end{matrix}\right)$, respectively, see Remark \ref{rem:epxlitchoices}. 
Hence, if a smooth cubic fourfold $X$ defined a point in $\kc_6$, then $H^{2,2}(X,\ZZ)_{\rm pr}$ would contain
a class $\delta$ with $(\delta)^2=2$ contradicting \cite[\S 4, Prop.\ 1]{VoisinGT}. In \cite{HassComp} one finds
an argument using limiting mixed Hodge structures to also exclude the case $[X]\in \kc_2$. So, $M\subset \kc\setminus(\kc_2\cup\kc_6)$.

\begin{thm}[Laza, Looijenga]
The period map identifies the moduli space $M$ of smooth cubic fourfolds with the
complement of $\kc_2\cup\kc_6$:
$$\kp\colon M\congpf\kc\setminus(\kc_2\cup\kc_6).$$
\end{thm}

To complete the picture, we state the following result. We refrain from giving a proof, but refer to similar
results in the theory of K3 surfaces \cite[Prop.\ 6.2.9]{HuyK3}.

\begin{prop}
The union $\bigcup\kc_d\subset\kc$ of all $\kc_d$ with $d$ satisfying $(\ast\!\ast\!\ast)$ is analytically dense in $\kc$.
Consequently, the union of all $\kc_d$ for satisfying $(\ast\ast')$ (or $(\ast\ast)$ or $(\ast))$ is analytically dense.
\end{prop}

 \begin{remark} On the level of moduli spaces, the theory of K3 surfaces is linked
 with the theory of cubic fourfolds in terms of the morphism $$\Phi_\varepsilon\colon M_d\subset\km_d\to \kc_d\subset\kc,$$
cf.\ Corollary \ref{cor:piceps}. Note that the image of a point $[(S,L)]\in M_d$ corresponding to a polarized K3 surface $(S,L)$ can a priori
be contained in the
boundary $\kc\setminus M=\kc_2\cup \kc_6$. However, unless $d=2$ or $d=6$, generically this is not the case and the map defines a rational
map $$\xymatrix{\Phi_\varepsilon\colon M_d\ar@{-->}[r]&M},$$ which is of degree one or two. 
\end{remark}

\subsection{} In Section \ref{sec:HodgeAssoK3C} we have linked Hodge theory of K3 surfaces and Hodge theory of cubic fourfolds. We will now
cast this in the framework of period maps and moduli spaces, i.e.\ in terms of the maps $\Phi_\varepsilon$.

\begin{prop}
A smooth cubic fourfold $X$ and a polarized K3 surface $(S,L)$ are associated, $(S,L)\sim X$, in the sense
of Definition \ref{def:assoK3C} if and only if $\Phi_\varepsilon[(S,L)]=[X]$ for some choice of $\varepsilon\colon\Gamma_d\congpf\Lambda_d$:
$$(S,L)\sim X~\Leftrightarrow ~Ê\exists\,Ê\varepsilon\colon\Phi_\varepsilon[(S,L)]=[X].$$
\end{prop}

\begin{proof}
Assume $\Phi_\varepsilon[(S,L)]=[X]$.
Pick an arbitrary marking
$H^2(S,\ZZ)\congpf\Lambda$ with $L\mapsto \ell$. Composing the  induced iso\-metry $H^2(S,\ZZ)_{L\text{\rm -pr}}\congpf\Lambda_d$ with $\varepsilon^{-1}\colon\Lambda_d\congpf\Gamma_d\subset\Gamma$ yields a point in $D_d\subset D$. 
Then there exists a marking $ H^4(X,\ZZ)_{\rm pr}\cong \Gamma$ such that $X$ yields the same period point in $D$,
which thus yields  a Hodge isometric embedding $H^2(S,\ZZ)_{L\text{\rm -pr}}\,\hookrightarrow H^4(X,\ZZ)_{\rm pr}(1)$.
Conversely, any such Hodge isometric embedding defines
a sublattice of $\Gamma\cong H^4(X,\ZZ)_{\rm pr}$ isomorphic to some $v^\perp$ which after applying some element in $\OO(\Gamma)$ becomes $\Gamma_d$, see Proposition \ref{prop:HassettEichler}. Composing with a marking of $(S,L)$ yields the appropriate $\varepsilon$.
\end{proof}

\begin{cor}\label{cor:numchar}
Let $X$ be a smooth cubic fourfold.
\begin{enumerate}
\item[{\rm (i)}] For fixed $d$, there exists a  polarized K3 surface $(S,L)$ of degree $d$ with  $X\sim (S,L)$ if and only if $X\in\kc_d$ 
and $d$ satisfies $(\ast\ast)$.
\item[{\rm (ii)}] There exists a twisted K3 surface $(S,\alpha)$ with $X\sim (S,\alpha)$ if and only if $X\in\kc_d$ for
some $d$ satisfying $(\ast\ast')$.
\end{enumerate}
\end{cor}

\begin{proof}
Consider $\km_d$ as the moduli space of quasi-polarized K3 surfaces $(S,L)$, i.e.\ with $L$ only big and nef but not necessarily
ample. One then has to show that
whenever  there exists a Hodge isometric embedding $H^2(S,\ZZ)_{L\text{\rm -pr}}\,\hookrightarrow H^4(X,\ZZ)_{\rm pr}(1)$, then $L$ is not orthogonal
to any algebraic class $\delta_S \in H^2(S,\ZZ)$ with $(\delta_S)^2=-2$. Indeed,  in this case $L$ would be automatically ample. However, such a class
$\delta_S$ would correspond to a class $\delta\in H^{2,2}(X,\ZZ)_{\rm pr}$ with $(\delta)^2=2$, which contradicts $[X]\in M=\kc\setminus(\kc_2\cup\kc_6)$.
Of course, the argument is purely Hodge theoretic and one can easily avoid talking about quasi-polarized K3 surfaces.

To prove  (ii), observe that the period of $X$ is contained in  $D_d$ if and only if
one finds $L_d\,\hookrightarrow \widetilde H^{1,1}(X,\ZZ)$. If $d$ satisfies $(\ast\ast')$, then there exists $U(n)\,\hookrightarrow L_d$
and we can conclude by Proposition \ref{prop:UUn}. Conversely, if $(S,\alpha)\sim X$, one finds $U(n)\,\hookrightarrow\widetilde H^{1,1}(S,\alpha,\ZZ)\cong \widetilde H^{1,1}(X,\ZZ)$.
As there also exists a positive plane $A_2\,\hookrightarrow \widetilde H^{1,1}(X,\ZZ)$, the lattice $U(n)$ is contained
in a primitive sublattice of rank three in $H^{1,1}(X,\ZZ)$, which is then necessarily of the form $L_d$ for some $d$ satisfying $(\ast\ast')$.
\end{proof}

A geometric interpretation of the condition $(\ast\!\ast\!\ast)$, involving the Fano variety of lines $F(X)$, will be explained
in the next section, see Proposition \ref{prop:AddHasHuy}. The conditions $(\ast\ast)$ and $(\ast\ast')$ will occur there again as well.

\begin{remark}
Note that a given cubic fourfold $X$ can be associated with more than one polarized K3 surface $(S,L)$ and, in fact, sometimes even
with infinitely many $(S,L)$. To start,
there are the finitely many choices of $\varepsilon\in\OO(\Lambda_d)/\tilde\OO(\Lambda_d)$, see \cite[Thm.\ 5.2.3]{HassComp}. Then, $\Phi_\varepsilon$
is only generically injective for $d$ satisfying  $(\ast\ast)_2$ and even of degree two for $(\ast\ast)_0$. And finally, $X$ could
be contained in more than one $\kc_d$. In fact, it can happen that $X\in\kc_d$ for infinitely many $d$ satisfying $(\ast\ast)$.
To be more precise, depending on the degree $d$, there may exist non-isomorphic K3 surfaces $S$ and $S'$ endowed
with polarizations $L$ and $L'$, respectively, such there nevertheless
exists a Hodge isometry $H^2(S,\ZZ)_{L\text{\rm -pr}}\cong H^2(S',\ZZ)_{L'\text{\rm -pr}}$. Indeed, the latter may not extend to a Hodge isometry $H^2(S,\ZZ)\cong H^2(S',\ZZ)$, see Section \ref{sec:HodgeAssoK3C}.

The situation is not quite as bad as it sounds. Although there may be infinitely many polarized K3 surfaces $(S,L)$ associated with one
$X$, only finitely many isomorphism types of unpolarized K3 surfaces $S$ will be involved.
\end{remark}

\begin{remark}
In \cite{Brakkee} a geometric interpretation for the generic fibre of the rational map $\Phi_\varepsilon\colon M_d\to \kc_d$ in the case
$d\equiv0\,(6)$ is described. It turns out that $\Phi_\varepsilon[(S,L)]=[(S',L')]$ implies that $S'$ is isomorphic to $M(3,L,d/6)$, the moduli
space of stable bundles on $S$ with the indicated Mukai vector.
\end{remark}

\section{Fano perspective}
We come back to the Hodge structure $v(\lambda_1)^\perp\subset\widetilde H(X,\ZZ)$,
see Remarks \ref{rem:l1perp} and \ref{rem:KuznH}. To give it a geometric interpretation,
we consider the Fano correspondence
\begin{equation}\label{eqn:Fanocorr}
\xymatrix{F(X)&\ar[l]_-p {\mathbb L}\ar[r]^-q&X.}
\end{equation}
Here, $F(X)$ is the Fano variety of lines contained in $X$, $p\colon\LL\to F(X)$ is the universal line, and $q$ is the natural projection,
cf.\ \cite[Ch.\ 3]{HuyCub} for details and references.
Due to work of Beauville and Donagi \cite{BD}, it is known that $F(X)$ is a four-dimensional
hyperk\"ahler manifold deformation equivalent to the Hilbert scheme $S^{[2]}$ of a K3 surface $S$.

\subsection{} The fact that $F(X)$ is of ${\rm K3}^{[2]}$-type implies that $H^2(F(X),\ZZ)$ with the Beauville--Bogomolov pairing
is isometric to the lattice $H^2(S^{[2]},\ZZ)\cong\Lambda\oplus\ZZ(-2)$. But the cohomology of the Fano variety can also be 
compared to $\widetilde H(X,\ZZ)$ by the following combination of \cite{AddCub,BD}.

\begin{thm}[Beauville--Donagi, Addington]
The Fano correspondence (\ref{eqn:Fanocorr}) induces two compatible Hodge isometries
$$\begin{array}{lcl}
H^4(X,\ZZ)_{\rm pr}(1)&\congpf& H^2(F(X),\ZZ)_{\rm pr}\\
\phantom{HH}\bigcap&&\phantom{HH}\bigcap\\
\phantom{H,}v(\lambda_1)^\perp&\congpf& H^2(F(X),\ZZ)\\
\phantom{HH}\bigcap&&\phantom{HH}\\
~~\widetilde H(X,\ZZ).&&
\end{array}$$
\end{thm}
On the left hand side, $H^4(X,\ZZ)_{\rm pr}(1)\subset v(\lambda_1)^\perp\subset \widetilde H(X,\ZZ)$ is the Hodge structure
introduced earlier on the sublattice $v(\lambda_1)^\perp\cong\lambda_1^\perp\cong \Lambda\oplus\ZZ(-2)$.
As before, the sign of the intersection pairing on $H^4(X,\ZZ)_{\rm pr}$ is changed.
On the right hand side, $H^2(F(X),\ZZ)_{\rm pr}$ is the primitive cohomology with respect to the Pl\"ucker polarization
$g\in H^2(F(X),\ZZ)$. It is endowed with a natural quadratic form, the Beauville--Bogomolov form on the hyperk\"ahler fourfold $F(X)$.
We shall not attempt to prove the result but we will  define  the maps that are used and indicate the main steps of
 the argument. 
 
 First, it has been observed in \cite{BD} 
that $$\varphi\coloneqq p_*\circ q^*\colon H^4(X,\ZZ)(1)\to H^2(F(X),\ZZ)$$ maps $h^2$ to the Pl\"ucker polarization $g\in H^2(F(X),\ZZ)$
and that for four-dimensional cubics the map induces an isomorphism
$$H^4(X,\ZZ)_{\rm pr}(1)\congpf H^2(F(X),\ZZ)_{\rm pr}$$
of Hodge structures of weight two satisfying $(\alpha)^2=-\frac{1}{6}\int_{F(X)} \varphi(\alpha)^2\cdot g^2$, cf.\ \cite[Sec.\ 3.4]{HuyCub} for  statements and further references.

Now, as $v(\lambda_1)^\perp\subset\widetilde H(X,\ZZ)\subset H^*(X,\QQ)$ is not concentrated in degree four, we need to extend the above to the full cohomology.
As was observed by Mukai, the natural map $p_*\circ q^*$ needs to be modified to enjoy certain functoriality properties. More precisely, it is known that
the following diagram commutes
\begin{equation}\label{eqn:MukaiCom}
\xymatrix{K_{\rm top}(X)\ar[d]_{v}\ar[r]^{p_*\circ q^*}&K_{\rm top}(F(X))\ar[d]^{v}\\
H^*(X,\QQ)\ar[r]&H^*(F(X),\QQ).}
\end{equation}
Here, the top and bottom rows are given by $E\mapsto p_*(q^*E)$ and $\alpha\mapsto p_*(q^*\alpha\cdot v(i_*\ko_\LL))$, respectively,
where $i\colon \LL\subset X\times F(X)$ is the inclusion, see \cite[Ch.\ 5]{HuyFM}.
The Mukai vector $i_*\ko_\LL$ can be computed by means of the Grothendieck--Riemann--Roch formula
as $$v(i_*\ko_\LL)=i_*({\rm td}(p))\cdot \left({{\rm td}(X)}^{-1}\boxtimes {\rm td}(F(X))\right)^{1/2}.$$
From here it is a straightforward computation to show that the commutativity of the diagram (\ref{eqn:MukaiCom})
implies the commutative diagram 
$$\xymatrix{K_{\rm top}(X)\ar[d]_{{\rm ch}}\ar[r]^{p_*\circ q^*}&K_{\rm top}(F(X))\ar[d]^{{\rm ch}}\\
H^*(X,\QQ)\ar[r]_-\varphi&H^*(F(X),\QQ),}$$
where now the bottom row is defined as $\varphi\colon\alpha\mapsto p_*(q^*\alpha\cdot{\rm td}(p))$.
In particular, for any class $\gamma\in K_{\rm top}(X)$ one finds ${\rm c}_1(p_*(q^*(\gamma)))=\{p_*(q^*{\rm ch}(\gamma)\cdot{\rm td}(p))\}_2$.

The restriction of ${\rm c}_1\circ p_*\circ q^*\colon K_{\rm top}(X)\to H^2(F(X),\ZZ)$ to the primitive part $A_2^\perp\subset K'_{\rm top}(X)$, i.e.\ the part mapping to $H^4(X,\QQ)_{\rm pr}$ under ${\rm ch}$ (or, equivalently, under the Mukai vector $v$), factors over the original isometry
$H^4(X,\ZZ)_{\rm pr}(1)\congpf H^2(F(X),\ZZ)_{\rm pr}$. As observed in Remark \ref{rem:l1perp}, $\lambda_1^\perp\subset  K'_{\rm top}(X)$
contains $A_2^\perp\oplus\ZZ\,(\lambda_1+2\lambda_2)$ as a sublattice of index three. A computation reveals  where the second summand
is mapped to, cf.\ \cite{AddCub}.

\begin{lem}[Addington]
Under the map ${\rm c}_1\circ p_*\circ q^*\colon K_{\rm top}(X)\to H^2(F(X),\ZZ)$ the class $\lambda_1+2\lambda_2$ is mapped to the Pl\"ucker polarization $g\in H^2(F(X),\ZZ)$. Furthermore, $(\lambda_1+2\lambda_2)^2=(g)^2=6$, where the second square  is with respect to the Beauville--Bogomolov form.\qed
\end{lem}

Therefore, there exists an isometric embedding of the sublattice
\begin{equation}\label{eqn:IndexthreeembedFano}
A_2^\perp \oplus\ZZ\,(\lambda_1+2\lambda_2)\,\hookrightarrow H^2(F(X),\ZZ),
\end{equation}
where $A_2^\perp \oplus\ZZ\,(\lambda_1+2\lambda_2)$ is a sublattice of $\lambda_1^\perp$ of index three and discriminant ${\rm disc}=18$.
On the other hand, as abstract lattices $H^2(F(X),\ZZ)\cong \lambda_1^\perp$. Using this, one then proves that (\ref{eqn:IndexthreeembedFano}) indeed extends to an isometry $\lambda_1^\perp\congpf H^2(F(X),\ZZ)$. Composition with
$\lambda_1^\perp\cong v(\lambda_1)^\perp$ yields the Hodge isometry $v(\lambda_1)^\perp\congpf H^2(F(X),\ZZ)$.
Here, the orthogonal complements $\lambda_1^\perp$ and $v(\lambda_1)^\perp$ are taken in $K_{\rm top}'(X)$
and $\widetilde H(X,\ZZ)$, respectively.

\subsection{}
In the sequel, we will think of $H^2(F(X),\ZZ)$ as a natural sub Hodge structure of $\widetilde H(X,\ZZ)$:
$$H^2(F(X),\ZZ)\subset \widetilde H(X,\ZZ),$$ orthogonal to the distinguished class $v(\lambda_1)\in \widetilde H^{1,1}(X,\ZZ)$. This should be thought
of as analogous to the inclusion $$H^2(S^{[2]},\ZZ)\subset \widetilde H(S,\ZZ),$$
which is orthogonal to $v(\ki_x)=(1,0,-1)\in (H^0\oplus H^4)(S,\ZZ)\subset \widetilde H^{1,1}(S,\ZZ)$. Note that both vectors, $v(\lambda_1)$ and $v(\ki_x)$, are of square two, which immediately leads to the following observation.

\begin{lem}
Let $X$ be a smooth cubic fourfold and $S$ a K3 surface. Then every Hodge isometry
$H^2(F(X),\ZZ)\congpf H^2(S^{[2]},\ZZ)$ extends to a Hodge isometry $\widetilde H(X,\ZZ)\congpf \widetilde H(S,\ZZ)$ mapping
$v(\lambda_1)$ to $v(\ki_x)$.\qed
\end{lem}
The result should be compared to the observation made earlier that every Hodge isometry
$H^4(X,\ZZ)_{\rm pr}\congpf H^4(X',\ZZ)_{\rm pr}$ extends to $H^4(X,\ZZ)\congpf H^4(X',\ZZ)$ with $h_X^2\mapsto \pm h_{X'}^2$.
\smallskip

This enables one to prove the Fano analogue of Proposition \ref{prop:UUn}, see \cite{AddCub,HassComp,HuyCub}.
\begin{prop}[Addington, Hassett, Huybrechts]\label{prop:AddHasHuy} Assume $X$ is a smooth cubic fourfold.
\begin{enumerate}
\item[{\rm (i)}] There exist a K3 surface $S$ and a Hodge isometry 
\begin{equation}\label{eqn:Prop1}
H^2(S^{[2]},\ZZ)\cong H^2(F(X),\ZZ)
\end{equation}
if and only if there
exists an embedding $U\,\hookrightarrow \widetilde H^{1,1}(X,\ZZ)$ with $v(\lambda_1)$ contained in its image.
\item[{\rm (ii)}] There exist a K3 surface $S$ and a Hodge isometry 
\begin{equation}\label{eqn:Prop2}
H^2(M_S(v),\ZZ)\cong H^2(F(X),\ZZ)
\end{equation} for some
 smooth, projective, four-dimensional moduli space $M_S(v)$ of stable sheaves on $S$ if and only if there exists a K3 surface
 $S$ with $S\sim X$ if and only if
there exists an embedding $U\,\hookrightarrow \widetilde H^{1,1}(X,\ZZ)$.
\item[{\rm (iii)}] There exist a twisted K3 surface $(S,\alpha)$ and a Hodge isometry 
\begin{equation}\label{eqn:Prop3}
H^2(M_{S,\alpha}(v),\ZZ)\cong H^2(F(X),\ZZ)
\end{equation} for some
smooth, projective, four-dimensional moduli space $M_{S,\alpha}(v)$ of twisted stable sheaves on $S$ if and only if 
there exists a twisted K3 surface $(S,\alpha)$ with $(S,\alpha)\sim X$ if and only if
there exists an embedding $U(n)\,\hookrightarrow \widetilde H^{1,1}(X,\ZZ)$ for some $n\ne0$.
\end{enumerate}
\end{prop}

\begin{proof}
Any Hodge isometry (\ref{eqn:Prop1}) extends to a Hodge isometry $\widetilde H(S,\ZZ)\cong H^2(F(X),\ZZ)$ with $(1,0,-1)\mapsto v(\lambda_1)$.
As $(1,0,-1)\in U\cong (H^0\oplus H^4)(S,\ZZ)\subset \widetilde H^{1,1}(S,\ZZ)$, this proves one direction in (i). For the other direction
use the arguments in the proof of Proposition \ref{prop:UUn} to show that there exists a K3 surface $S$ with $S\sim X$ such that the given
$U\,\hookrightarrow\widetilde H^{1,1}(X,\ZZ)$ corresponds to $(H^0\oplus H^4)(S,\ZZ)$.

For (ii) and (iii) recall that there exists a Hodge isometry $H^2(M_{S,\alpha}(v),\ZZ)\cong v^\perp\subset\widetilde H(S,\alpha,\ZZ)$, cf.\ 
\cite[Ch.\ 10]{HuyK3} for references in the untwisted case and \cite{HuyStel} for the twisted case. Then, if a Hodge isometry $\widetilde H(S,\alpha,\ZZ)\cong \widetilde H(X,\ZZ)$ is given, let $v\in \widetilde H^{1,1}(S,\alpha,\ZZ)$ be the vector
that is mapped to $v(\lambda_1)$. Then (\ref{eqn:Prop2}) and (\ref{eqn:Prop3}) hold. The remaining assertions follow from Proposition  \ref{prop:UUn}.
\end{proof}

This leads to the following analogue of Corollary \ref{cor:numchar}.

\begin{cor}
For a smooth cubic fourfold $X$ the condition {\rm (i)} (or {\rm (ii)} or {\rm (iii)}) is equivalent to
$X\in\kc_d$ for some $d$ satisfying $(\ast\!\ast\!\ast)$ (or $(\ast\ast)$ or $(\ast\ast')$, respectively).\vskip-0.5cm\qed
\end{cor}
So, at one glance:
$$\begin{array}{rclcc}
H^2(S^{[2]},\ZZ)&\cong& H^2(F(X),\ZZ)&\Leftrightarrow& (\ast\!\ast\!\ast),\\
H^2(M_S(v),\ZZ)&\cong &H^2(F(X),\ZZ)&\Leftrightarrow  &(\ast\ast),\\
H^2(M_{S,\alpha}(v),\ZZ)&\cong &H^2(F(X),\ZZ)&\Leftrightarrow&  (\ast\ast').
\end{array}$$

\subsection{}\label{sec:GTHK4}
The purely Hodge and lattice theoretic considerations above can now be combined with the global Torelli theorem for hyperk\"ahler fourfolds
due to Verbitsky \cite{Verb} and Markman \cite{Mark}, see also \cite{HuyVerb}:
Two hyperk\"ahler fourfolds $Y$ and $Y'$ of ${\rm K3}^{[2]}$-type are birational if
and only if there exists a Hodge isometry
$H^2(Y,\ZZ)\cong H^2(Y',\ZZ)$:
$$Y\sim Y'\Leftrightarrow H^2(Y,\ZZ)\cong H^2(Y',\ZZ).$$
This then implies the following reformulation of the above results:
$$S^{[2]}\sim F(X)\Leftrightarrow (\ast\!\ast\!\ast), ~~ÊM_S(v)\sim F(X)\Leftrightarrow  (\ast\ast)$$
$$\text{\rm and~~} M_{S,\alpha}(v)\sim F(X)\Leftrightarrow  (\ast\ast').$$
More precisely, one has:

\begin{cor}\label{cor:FanoHilbModuli}
Let $X$ be a smooth cubic fourfold and $F(X)$ its Fano variety of lines. 
\begin{enumerate}
\item[{\rm (i)}] There exists a K3 surface $S$ such that $F(X)$ is birational to $S^{[2]}$ if and only if $X\in \kc_d$ for some $d$ satisfying $(\ast\!\ast\!\ast)$.
\item[{\rm (ii)}] There exists a K3 surface $S$ such that $F(X)$ is birational to a certain smooth, projective moduli
space $M_S(v)$ of stable sheaves on $S$ if and only if $X\in \kc_d$ for some $d$ satisfying $(\ast\ast)$.
\item[{\rm (iii)}] There exists a twisted K3 surface $(S,\alpha)$ such that $F(X)$ is birational to a certain smooth, projective moduli
space $M_{S,\alpha}(v)$ of twisted stable sheaves on $S$ if and only if $X\in \kc_d$ for some $d$ satisfying $(\ast\ast')$.\qed
\end{enumerate}
\end{cor}

\begin{remark}
For $d\equiv0\, (6)$ and very general $(S,L)\in M_d$, i.e.\ $\Pic(S)\cong \ZZ\, L$, there exists exactly one other polarized
K3 surface $(S',L')\in M_d$ with $\Phi_\varepsilon [(S,L)]=\Phi_\varepsilon[(S',L')]\eqqcolon [X]$. In particular, the Fano variety  $F(X)$ of lines
in the corresponding cubic fourfold $X$ is a natural four-dimensional hyperk\"ahler manifold associated with
$(S,L)$ and $(S',L')$. Other hyperk\"ahler manifolds that come naturally with $S$ and $S'$ would be $S^{[2]}$ and $S'^{[2]}$. From Corollary
\ref{cor:FanoHilbModuli} we know that for $d$ not satisfying $(\ast\!\ast\!\ast)$ the Hilbert scheme $S^{[2]}$ and the Fano variety $F(X)$ are not isomorphic. It was recently shown in \cite{Brakkee} that also $S^{[2]}$ and $S'^{[2]}$ need not be isomorphic (nor birational).
More precisely, they are isomorphc if and only if the Pell equation $3p^2-(d/6)q^2=-1$ has an integral solution.
\end{remark} 

\section{The Hodge theory of Kuznetsov's category}

In this short last section we touch upon the Hodge theoretic aspects of Kuznetsov's
triangulated category $\ka_X$ naturally associated with every smooth cubic fourfold $X\subset\PP^5$.
For the more categorical aspects we refer to the original \cite{Kuz1,Kuz2} or the lecture notes in this volume \cite{MacSt}.
The Hodge theoretic investigation of $\ka_X$ was initiated by Addington and Thomas \cite{AT}, the algebraic part of it played a crucial role
already in  \cite{Kuz2}.
 
\subsection{} We consider the bounded derived category $\Db(X)=\Db({\rm Coh}(X))$ of the abelian category ${\rm Coh}(X)$
of coherent sheaves on $X$. The three line bundles $\ko_X,\ko_X(1),\ko_X(2)\in \Db(X)$ form an exceptional
collection, i.e. $\Hom(\ko_X(i),\ko(j)[\ast])=0$ for $i>j$ and $\CC[0]$  for $i=j$.
According to a result of Bondal and Orlov \cite{BoOr:ICM}, the derived category $\Db(X)$ determines $X$ uniquely.
More precisely, if there exists an exact, linear equivalence $\Db(X)\cong \Db(X')$ for two smooth cubic fourfolds
$X,X'\subset\PP^5$, then $X\cong X'$. This could be called a categorical global Torelli theorem, although the existence
of such an equivalence is almost as hard as writing down an explicit isomorphism between them. However, it turns out that
$\Db(X)$ contains a natural subcategory which is a much subtler invariant.

\begin{definition}
For a smooth cubic fourfold $X\subset\PP^5$, we denote by
$$\ka_X\coloneqq\langle\ko_X,\ko_X(1),\ko_X(2)\rangle^\perp\subset\Db(X)$$
the full triangulated subcategory of all objects $F\in \Db(X)$ 
right orthogonal to $\ko_X,\ko_X(1)$, and $\ko_X(2)$, i.e.\ such that $\Hom(\ko_X(i),F[\ast])=0$ for $i=0,1,2$.
\end{definition}

\begin{thm}[Kuznetsov]
The triangulated category $\ka_X$ is a Calabi--Yau  category of dimension two, i.e.\ $F\mapsto F[2]$ defines a Serre functor.\qed
\end{thm}

In other words, for all $E,F\in\ka_X$ there exist functorial isomorphisms $$\Hom(E,F)\cong\Hom(F,E[2])^\ast.$$ Other examples
of such categories are provided by $\Db(S)$ and $\Db(S,\alpha)$ associated with K3 surfaces $S$ and twisted K3 surfaces $(S,\alpha)$.
A natural question in this context is now to determine  when the Kuznetsov category
$\ka_X$ associated with a cubic fourfold is equivalent to the derived category $\Db(S)$ or $\Db(S,\alpha)$ for some
(twisted) K3 surface.

\subsection{} The goal of \cite{AT} was to compare Hassett's condition $(\ast\ast)$ with the condition
$\ka_X\cong\Db(S)$. Building upon \cite{AT}, the twisted version was later dealt with in \cite{HuyComp}.

\begin{thm}[Addington--Thomas, Huybrechts]
Let $X$ be a smooth cubic fourfold and $(S,\alpha)$ a twisted K3 surface.
\begin{enumerate}[(ii)]
\item[{\rm (i)}] Any exact, linear equivalence $\ka_X\cong \Db(S)$ induces a Hodge isometry $\widetilde H(X,\ZZ)\cong \widetilde H(S,\ZZ)$. In particular, $X$ is contained in $\kc_d$ with $d$ satisfying $(\ast\ast)$.
\item[{\rm (ii)}] Any exact, linear equivalence $\ka_X\cong\Db(S,\alpha)$ induces a Hodge isometry   $\widetilde H(X,\ZZ)\cong \widetilde H(S,\alpha,\ZZ)$. In particular, $X$ is contained in $\kc_d$ with $d$ satisfying $(\ast\ast')$.
\end{enumerate}
\end{thm}

In fact, it is also known that for very general $X\in \kc_d$ with $d$ satisfying $(\ast\ast)$ or $(\ast\ast')$, respectively, the
converse in (i) and (ii) hold true. The proof, however, requires a fair amount of deformation theory for Fourier--Mukai kernels developed in
\cite{HMS,Toda,AT,HuyComp}. For non-special cubic fourfolds one has the following result.

\begin{prop}[Huybrechts]
Let $X$ and $X'$ be smooth cubic fourfolds. Then any Fourier--Mukai equivalence
$\ka_X\cong\ka_{X'}$ induces a Hodge isometry $\widetilde H(X,\ZZ)\cong\widetilde H(X',\ZZ)$. 
The converse holds for all non-special $X$ and for general special ones.
\end{prop}

The results of the forthcoming \cite{BLMetal} complete this picture, so that eventually we will have
$$\begin{array}{lcl}
\ka_X\cong\Db(S)&\Leftrightarrow& \widetilde H(S,\ZZ)\cong \widetilde H(X,\ZZ) \text{ Hodge isometry},\\
\ka_X\cong\Db(S,\alpha)&\Leftrightarrow& \widetilde H(S,\alpha,\ZZ)\cong \widetilde H(X,\ZZ) \text{ Hodge isometry},\\
\ka_X\cong\ka_{X'}&\Leftrightarrow&\widetilde H(X,\ZZ)\cong \widetilde H(X',\ZZ)\text{ Hodge isometry}.
\end{array}$$


\end{document}